\documentclass{article}
\usepackage{amsmath}
\usepackage{amsthm}
\usepackage{amssymb,epsfig}
\usepackage[cp 1250]{inputenc}
\usepackage{a4}
\usepackage{graphics}
\usepackage[dvips]{color}

\newtheorem{theorem}{Theorem}[section]
\newtheorem{lemma}[theorem]{Lemma}
\newtheorem{proposition}[theorem]{Proposition}
\newtheorem{corollary}[theorem]{Corollary}

\theoremstyle{definition}
\newtheorem{remark}[theorem]{Remark}

\newtheorem{definition}[theorem]{Definition}
\newtheorem{claim}[theorem]{Claim}

\newcommand{\lexp}[1]{\vphantom{a}^{#1}}

\def\Z{\mathbb Z}
\def\N{\mathbb N}

\def\A{\mathcal A}

\def\pf{\begin{proof}}
\def\pfk{\end{proof}}
\def\le{\leqslant}
\def\leq{\leqslant}
\def\ge{\geqslant}
\def\geq{\geqslant}

\begin{document}


\title{ }

\begin{center}
\vskip 1cm{\LARGE\bf Minimal digit sets for parallel addition \\
\vskip .02in
in non-standard numeration systems}
\vskip 1cm
\large Christiane Frougny\\
LIAFA, CNRS UMR 7089\\
Case 7014, 75205 Paris Cedex 13\\
France\\
{\tt christiane.frougny@liafa.univ-paris-diderot.fr}\\
\ \\
Edita Pelantov\'a\\ 
Doppler Institute for Mathematical Physics and Applied Mathematics, and Department of Mathematics\\
Czech Technical University in Prague\\
Trojanova 13, 120 00 Praha 2\\
Czech Republic\\
{\tt edita.pelantova@fjfi.cvut.cz}\\
\ \\
Milena Svobodov\'a\\
Doppler Institute for Mathematical Physics and Applied Mathematics, and Department of Mathematics\\
Czech Technical University in Prague\\
Trojanova 13, 120 00 Praha 2\\
Czech Republic\\
{\tt milenasvobodova@volny.cz}\\
\end{center}

\vskip .2 in

\begin{center} {\em Dedicated to Jean-Paul Allouche for his sixtieth birthday} \end{center}

\vskip .2 in
\begin{abstract}

We study parallel algorithms for addition of numbers having finite
representation in a positional numeration system defined by a base
$\beta$ in $\mathbb{C}$ and a finite digit set $\mathcal{A}$ of
contiguous integers containing $0$. For a fixed base $\beta$, we
focus on the question of the size of the alphabet allowing to
perform addition in constant time independently of the length of
representation of the summands. We produce lower bounds on the
size of such alphabet $\mathcal{A}$. For several types of well
studied bases (negative integer, complex numbers $ -1 + \imath$,
$2 \imath$, and $\imath \sqrt{2}$, quadratic Pisot unit, and the
non-integer rational base), we give explicit parallel algorithms
performing addition in constant time. Moreover we show that digit
sets used by these algorithms are the smallest possible.

\end{abstract}



\section{Introduction}

Since the beginnings of computer science, the fact that addition
of two numbers has a worst case linear time complexity has been
considered as an important
 drawback, see in particular the seminal paper of Burks,
  Goldstine and Von Neumann \cite{BGVN}. In 1961, Avizienis
  gave a parallel algorithm to add two numbers: numbers are
   represented in base $10$ with digits from the set $\{-6,-5,\ldots,5,6\}$,
   which allows no carry propagation~\cite{Avizienis}. Note that
   already Cauchy in 1840 considered the representation of numbers
   in base $10$ and digit set $\{-5,\ldots,5\}$, and remarked that
   carries have little propagation, due to the fact that positive
   and negative digits are mutually cancelling in the addition process, ~\cite{Cauchy}.

Since the Avizienis paper, parallel addition has received a lot of
attention, because it is the core of some fast multiplication and
division algorithms, see for instance~\cite{ErcegovacLang}.
General conditions on the digit set allowing parallel
 addition in positive integer base can be found in \cite{Parhami} and
\cite{Kornerup}.

A positional numeration system is given by a base and by a set of
digits. The base $\beta$ is a real or complex number such that
$|\beta|>1$, and the digit set $\mathcal{A}$ is a finite alphabet
of real or complex digits. Non-standard numeration systems ---
where the base $\beta$ is not a positive integer --- have been
extensively studied. When $\beta$ is a real number $>1$, this is
the well known theory of the so-called $\beta$-expansions due to
R\'enyi \cite{Renyi} and Parry \cite{Parry}. Special attention has
been paid to complex bases, which allow to represent any complex
number by a single sequence (finite or infinite) of natural
digits, without separating the real and the imaginary part. For
instance, in the Penney numeration system every complex number can
be expressed with base $-1+\imath$ and digit set $\{0,1\}$,
\cite{Penney}. The Knuth~\cite{Knuth} numeration system is defined
by the base $2 \imath$ with digit set $\{0,\ldots,3\}$. Another
complex numeration system with digit set $\{0,1\}$ is based on
$\imath \sqrt{2}$, see~\cite{NM}.
\\

For designing a parallel algorithm for addition, some
\emph{redundancy} is necessary. In the Avizienis or Cauchy
numeration systems, numbers may have several representations. In
order to have parallel addition on a given digit set, there must
be \emph{enough} redundancy, see~\cite{Mazenc}
and~\cite{Kornerup}. Both the Avizienis and the Cauchy digit sets
allow parallel addition, but
the Avizienis digit set is not minimal for parallel addition, as the Cauchy digit set is.\\

When studying the question on which digit sets it is possible to
do addition in parallel for a given base $\beta$,  we restrict
ourselves to the case that the digit set is an alphabet of
contiguous integer digits containing $0$. This assumption already
implies that the base $\beta$ is an algebraic number. In a
previous paper~\cite{FrPeSv}, we have shown that it is possible to
find an alphabet of integer digits on which addition can be
performed in parallel when $\beta$ is an algebraic number with no
algebraic conjugates of modulus $1$. This digit set is not minimal
in
general, but the algorithm is quite simple, it is a kind of generalization of the Avizienis algorithm.\\

In this work we focus on the problem of finding an alphabet of
digits allowing parallel addition that is minimal in size. The
paper is organized as follows:

We first give lower bounds on the cardinality of the minimal
alphabet allowing parallel addition. When $\beta$ is a real
positive algebraic number, the bound is $\lceil \beta \rceil$.
When $\beta$ is an algebraic integer with minimal polynomial
$f(X)$, the lower bound is equal to $|f(1)|$. This bound can be
refined to $|f(1)|+2$ when $\beta$ is a real positive algebraic
integer.

Addition on an alphabet $\A$ can be seen as a digit set conversion
between alphabets $\A+\A$ and $\A$. In
Section~\ref{sec:conversion}, we show that the problem of parallel
addition on $\A$ can be reduced to problems of parallel digit set
conversion between alphabets of cardinality smaller than the one
of $\A+\A$, Proposition~\ref{conversion}. We also give some method
allowing to link parallel addition on several alphabets of the
same cardinality, namely to transform an algorithm for parallel
addition on one alphabet into algorithms performing parallel
addition on other alphabets.

We then examine some popular numeration systems, and show that our
bounds are attained. When $\beta$ is an integer~$\ge 2$, our bound
comes to $\beta +1$, and it is known that parallel addition is
feasible on any alphabet of this size, which is minimal,
see~\cite{Parhami} for instance.

In the case that the base is a negative  integer, $\beta= -b$, $b
\ge 2$, the lower bound we obtain is again equal to $b+1$. We show
that parallel addition is possible not only on the alphabet $\{0,
\ldots, b\}$, but in fact on any alphabet (of contiguous integers
containing~$0$) of cardinality $b+1$.

We then consider the more general case where  the base has the
form $\beta = \sqrt[k]{b}$, $b \in \mathbb{Z}$, $|b|\ge 2$, and $k
\in \N$, $k \ge 1$. We show that parallel addition is possible on
every alphabet (of contiguous integers containing~$0$) of
cardinality $b+1$. If $b \ge 2$, then this cardinality is minimal
(assuming that the expression of $\beta = \sqrt[k]{b}$ is written
in the minimal form). We use this result on several examples: The
complex base $\beta=-1+\imath$ satisfies $\beta^4=-4$, and the
minimal alphabet for parallel addition must have $5$~digits, in
fact it can be any alphabet (of contiguous integers
containing~$0$) of cardinality~$5$. With similar reasoning, also
for the Knuth numeration system, with base $\beta=2 \imath$,
parallel addition is doable on any alphabet (of contiguous
integers containing~$0$) of cardinality~$5$. Analogously, in base
$\beta = \imath \sqrt{2}$ parallel addition is doable on any
alphabet (of contiguous integers containing~$0$) of
cardinality~$3$.

We then consider $\beta$-expansions, where $\beta$ is  a quadratic
Pisot unit, i.e., the largest zero of a polynomial of the
form $X^2-aX+1$, with $a \in \N$, $a\ge 3$, or of a polynomial of
the form $X^2-aX-1$, with $ a \in \N$, $a\ge 1$. Such numeration
systems have been extensively studied, since they enjoy a lot of
nice properties. In particular, by a greedy algorithm, any
positive integer has a finite $\beta$-expansion, and it is known
that the set of finite $\beta$-expansions is closed under
addition~\cite{BFGK}. In the case $\beta^2=a\beta-1$, any positive
real number has a $\beta$-expansion on the alphabet $\{0, \ldots,
a-1\}$. We show that every alphabet (of contiguous integers
containing~$0$) of cardinality $a$ is sufficient to achieve
parallel addition, so the lower bound $|f(1)|+2$ is reached. In
the case $\beta^2=a\beta+1$, any positive real number has a
$\beta$-expansion on the alphabet $\{0, \ldots, a\}$. We show that
parallel addition is possible on any alphabet (of contiguous
integers containing~$0$) of cardinality $a+2$, which also achieves
our lower bound $|f(1)|+2$. In both cases, we provide explicitly
the parallel algorithms.

One case where the base is not an algebraic integer, but an
algebraic number, is the rational number $\pm a/b$, with $a >b \ge
2$. When $\beta=a/b$ our bound is equal to $\lceil a/b \rceil$,
which is not good enough, since we show that the minimal alphabet
has cardinality $a+b$. We prove that parallel addition is doable
on $\{0, \ldots,a+b-1\}$, on the opposite alphabet
$\{-a-b+1,\ldots,0\}$, and on any alphabet of cardinality $a+b$
containing $\{-b,\ldots,0, \ldots,b\}$. In the negative case,
$\beta=-a/b$, our results do not provide a lower bound. We show
that the minimal alphabet has cardinality $a+b$, and any alphabet
of such cardinality allows parallel addition.

The question of determining the size of minimal alphabet for
parallel addition in other numeration systems remains open.


\section{Preliminaries}\label{Prel}


\subsection{Numeration systems}

For a detailed presentation on these topics, the reader
 may consult~\cite{cant}. A \textit{positional numeration system}
 $(\beta,\mathcal{A})$ within the complex field $\mathbb{C}$ is
 defined by a \textit{base} $\beta$, which is a complex number
  such that $|\beta|>1$, and a \textit{digit set}  $\mathcal{A}$
  usually called \textit{alphabet}, which is a subset of $\mathbb{C}$.
  In what follows, $\mathcal{A}$ is finite and contains $0$. If a complex
  number $x$ can be expressed in the form  $ \sum_{j\leq n} x_j\beta^j$
  with coefficients $x_j$ in $\mathcal{A}$, we call the sequence
  $(x_j)_{j\leq n}$ a $(\beta, \mathcal{A})$\textit{-representation} of $x$.\\

The problem of representability in a complex base is far from
being completely characterized, see the survey~\cite{cant}.
However, when the base is a real number, the domain has been
extensively studied. The most elaborated case is the one of
representations of real numbers in a non-integer base $\beta >1$,
the so-called {\em greedy expansions}, introduced by R\'enyi
\cite{Renyi}. Denote by $T$  a transformation $T:[0,1) \rightarrow
[0,1)$  given by the prescription
$$T(x)= \beta x - D(x), \ \ \hbox{where $D(x) = \lfloor \beta x\rfloor$}. $$
Then
$$x =\frac{D(x)}{\beta} + \frac{T(x)}{\beta} \quad \hbox{for any}\ \  x \in [0,1).$$
Since $ T(x) \in [0,1)$  as well, we can  repeat this process
infinitely many times, and thereby obtain a representation of $x
\in [0, 1)$ in the form
\begin{equation}\label{EqExpansion}
x =\frac{D(x)}{\beta} +\frac{D(T(x))}{\beta^2} +
\frac{D(T^2(x))}{\beta^3} +  \frac{D(T^3(x))}{\beta^4}+ \cdots
\end{equation}
This representation is called the {\em R\'enyi expansion} or {\em
greedy expansion} of $x$ and denoted $\langle x \rangle_\beta$.
Since the coefficients are  $D(x) = \lfloor \beta x\rfloor$ and $x
\in [0,1)$, the alphabet  of the R\'enyi expansion is
$\mathcal{C}_\beta = \{ 0,1, \ldots, \lceil\beta \rceil -1\}$. We
will refer to this alphabet as the {\em canonical alphabet} for
$\beta >1$. A sequence $(x_j)_{j\geq 1}$ such that $\langle x
\rangle_\beta = 0\bullet x_1x_2 x_3\cdots$ for some $x \in [0,1)$
is  called {\em $\beta$-admissible}. If this sequence has only
finitely many non-zero entries, we say that $x$ has a {\em finite}
 R\'enyi expansion in the base $\beta$. Let us stress that not all
 strings in the alphabet $\mathcal{C}_\beta $ are $\beta$-admissible.
  For characterization of $\beta $-admissible sequences see \cite{Parry}.
  If  the base $\beta$ is not an integer, then some numbers have more
  than one representation on the canonical
alphabet $\mathcal{C}_\beta $. It is important to mention that the
R\'enyi expansion  $\langle x \rangle_\beta$ is lexicographically
the greatest among all representations  $(x)_\beta$ over
$\mathcal{C}_\beta $.

In order to find a representation of a number $x \geq 1$, we can
use the R\'enyi transformation $T$ as well: At first we find a
minimal $k\in \mathbb{N}$ such that
 $y=x\beta^{-k} \in [0,1)$, then we determine
 $\langle y\rangle_\beta = 0\bullet y_1y_2y_3 \cdots$ and finally
 we put $\langle x\rangle_\beta = y_1y_2 \cdots y_k\bullet y_{k+1}y_{k+2} \cdots$.
If the base $\beta$ is an integer, say $\beta =10$, then the
R\'enyi expansion is the usual decimal expansion. The R\'enyi
expansion of a negative real number $x$ is defined as
$-\langle|x|\rangle_\beta$, which means that one additional bit
for the sign $\pm$ is necessary. In the R\'enyi expansion of
numbers (analogously to the decimal expansion), the algorithms for
addition and subtraction differ.

Since the R\'enyi transformation $T$ uses the alphabet
$\mathcal{C}_\beta$, we can represent any positive real number $x$
in this alphabet  as an infinite word $x_nx_{n-1}\cdots x_0\bullet
x_{-1}x_{-2}\cdots $. The numbers represented by finite prefixes
of this word tend to number $x$.

Let us now consider an integer $m$ satisfying $m <0< m+ \lceil
\beta\rceil - 1$, and an alphabet $\mathcal{A}_m = \{m, \ldots, 0,
\ldots, m+ \lceil \beta\rceil - 1\}$ of cardinality $\lceil
\beta\rceil$. Let
$$J_m= \bigl[\tfrac{m}{\beta -1},\tfrac{m}{\beta -1}+1\bigr).$$
We describe a transformation $T_m:J_m\rightarrow J_m$  which
enables to  assign to any real number $x$ a $(\beta,
\mathcal{A}_m)$-representation. Put
$$T_m(x) = \beta x -D_m(x),  \ \ \hbox{where $ D_m(x) =\bigl\lfloor \beta x -\tfrac{m}{\beta -1}\bigr\rfloor$}.$$
Since $T_m(x) -  \tfrac{m}{\beta -1} =   \beta x -\tfrac{m}{\beta
-1} - \bigl\lfloor \beta x -\tfrac{m}{\beta -1}\bigr\rfloor \in
[0,1)$, we have $T_m(x) \in \bigl[\tfrac{m}{\beta
-1},\tfrac{m}{\beta -1}+1\bigr)$ for any $x$ in $J_m$, and
therefore $T_m$ maps the interval  $J_m$ into $J_m$. Moreover, any
$x$ from the interval $J_m$ satisfies
$$ \beta x - \tfrac{m}{\beta -1} < \beta \bigl( \tfrac{m}{\beta -1} +1 \bigr)
 - \tfrac{m}{\beta -1} =  m +\beta \quad {\rm and} \quad \beta x
  - \tfrac{m}{\beta -1} \geq \tfrac{m\beta}{\beta -1} - \tfrac{m}{\beta -1} = m, $$
and thus $m\leq \bigl\lfloor \beta x -\tfrac{m}{\beta -1}\bigr\rfloor
\leq  m+ \lceil \beta\rceil - 1$, i.e., the digit $D_m(x)$
belongs to $\mathcal{A}_m$. Therefore, each $x$ in $J_m$ can be
written as in \eqref{EqExpansion}. Since  for any $x$ in $\mathbb{R}$
there exists a power $n$ in $\mathbb{N}$ such that $\frac{x}{\beta^n}$
 is in $J_m$, all real numbers have a $(\beta, \mathcal{A}_m)$-representation.
  This already implies that the set of numbers having finite $(\beta, \mathcal{A})$-representation is dense in $\mathbb{R}$.\\

Let us mention that, if we consider an alphabet  $\mathcal{A} $
such that $\mathcal{A} = -\mathcal{A} $, we can exploit instead of
$T_m$ a symmetrized  version of the R\'enyi algorithm introduced
by Akiyama and Scheicher in \cite{AkiSch}. They use the
transformation $S:[-\tfrac12,\tfrac12 ) \rightarrow
[-\tfrac12,\tfrac12 )$ given by the prescription $$ S(x)= \beta x
- D(x), \ \ \hbox{where $D(x) = \lfloor \beta x + \tfrac12
\rfloor$}.$$ This expansion  has again the form
\eqref{EqExpansion}, but the digit set is changed into
$$\mathcal{A} = \mathbb{Z} \cap
(-\tfrac{\beta+1}{2},\tfrac{\beta+1}{2}).$$ Since the alphabet is
symmetrical around $0$, it has an odd number of elements. In
general, it can be  bigger than the canonical alphabet
$\mathcal{C}_\beta$, but not too much, because $\lceil\beta \rceil
+1 \geq \# \mathcal{A} \geq \lceil\beta \rceil =\#
\mathcal{C}_\beta$. On the other hand, the Akiyama-Scheicher
representation has an important advantage: the representation of
$-x$ can be obtained from the representation of $x$ by replacing
the digit $a $ by the digit $-a$.  Therefore, an algorithm for
subtraction can exploit an algorithm for addition, and clearly, no
additional bit for indicating the sign is needed.

A  more general construction including our $T_m$ is discussed in \cite{KaSt}.\\

In the case of base $\beta$ being a rational number of the form
$a/b$, with $a > b\ge 1$, $a$ and $b$ co-prime, the greedy
algorithm gives a representation on the alphabet $\{0, \ldots,
\lceil a/b \rceil -1\}$, but another algorithm --- a modification
of the Euclidean division algorithm --- gives any natural integer
a unique and finite expansion on the alphabet $\{0, \ldots, a
-1\}$, see~\cite{FrougnyKlouda} and \cite{AkiFroSaka}.
For instance, if $\beta=3/2$, the expansion of the number $4$ is $21$.\\

Also the negative bases have been investigated.  Already in 1885,
negative integer base was described by Gr\"unwald, see
\cite{Grunwald}. When $\beta$ is a real number,
$(-\beta)$-expansions have been introduced in \cite{IS}. Negative
rational bases of the form $\beta = -a/b$, with $a > b\ge 1$, $a$
and $b$ co-primes, have been studied in~\cite{FrougnyKlouda}. Any
integer can be given a unique and finite expansion on the alphabet
$\{0, \ldots, a-1\}$ by a modification of the Euclidean division
algorithm, so this system is a \emph{canonical numeration system},
see~\cite{cant} for properties and results.


\subsection{Parallel addition}

We consider addition and subtraction in the set  of real or
complex numbers  from an algorithmic point of view. Similarly to
the classical algorithms for arithmetical operations, we work only
on the set of numbers with \emph{finite} representation, i.e., on the set
\begin{equation}\label{Fin_A_beta}
{\rm Fin}_{\mathcal{A}}(\beta) = \Bigl\{ \ \sum_{j\in I}\
{x_j\beta^j} \mid I \subset \mathbb{Z},  \ \ I \ \ \hbox{finite},
\ \ x_j \in \mathcal{A}\Bigr\}.
\end{equation}
Such a finite sequence $(x_j)_{j\in I}$ of elements of
$\mathcal{A}$ is  identified with a bi-infinite string
$(x_j)_{j\in \mathbb{Z}}$ in $\mathcal{A}^\mathbb{Z}$, wherein
only a finite number of digits $x_j$ have non-zero value. The
index zero in bi-infinite strings is indicated by $\bullet$. So if
$x$ belongs to ${\rm Fin}_{\mathcal{A}}(\beta)$, we write $$
(x)_{\beta, \mathcal{A}} =  \lexp{\omega}0 x_n x_{n-1}\cdots
x_1x_0\bullet x_{-1}x_{-2} \cdots x_{-s} 0^\omega $$ with
$x=\sum_{j=-s}^{j=n}\ {x_j\beta^j}$.

 Let $x, y
\in {\rm Fin}_{\mathcal{A}}(\beta)$, with $(y)_{\beta,
\mathcal{A}} = \lexp{\omega}0 y_ny_{n-1} \cdots y_1y_0\bullet
y_{-1} y_{-2} \cdots y_{-s} 0^\omega $. Adding $x$ and $y$ means
to rewrite the $(\beta,\mathcal{A}+\mathcal{A})$-representation
$$\lexp{\omega}0(x_n+y_n) \cdots (x_1+y_1)(x_0+y_0)\bullet (x_{-1} +y_{-1})
 \cdots (x_{-s}+y_{-s}) 0^\omega$$ of the number $x+y$ into a $(\beta, \mathcal{A})$-representation of $x+y$.

A necessary condition for existence of an algorithm rewriting
finite $(\beta,
 \mathcal{A}+\mathcal{A})$-representations into finite $(\beta,
 \mathcal{A})$-representations is that the set ${\rm Fin}_{\mathcal{A}}(\beta)$ is closed under addition, i.e.,
\begin{equation}\label{closednessUnderAddition}
{\rm Fin}_{\mathcal{A}}(\beta)+{\rm
Fin}_{\mathcal{A}}(\beta)\subset {\rm Fin}_{\mathcal{A}}(\beta).
\end{equation}
Let us point out that we are not specifically discussing in this
 paper whether or not the inclusion (\ref{closednessUnderAddition})
 is satisfied by a numeration system $(\beta,\mathcal{A})$; however,
 the numeration systems studied in that paper satisfy this inclusion.\\

As we have already announced, we are interested in parallel
algorithms for addition. Let us mathematically formalize
parallelism. Firstly,
 we recall the notion of a {\em local function}, which comes from
 symbolic dynamics (see \cite{LM}) and is often called a {\em sliding block code}.

\begin{definition}\label{local}
A function $\varphi : \mathcal{A}^{\Z} \rightarrow
\mathcal{B}^{\Z}$  is said to be {\em $p$-local} if there exist
two non-negative integers $r$ and $t$ satisfying $p=r+t+1$, and a
function $\Phi:\mathcal{A}^p\rightarrow \mathcal{B}$ such that,
for any $u=(u_j)_{j \in \Z} \in \mathcal{A}^{\Z}$ and its image
$v=\varphi(u)=(v_j)_{j\in \Z} \in \mathcal{B}^{\Z}$, we have
$v_{j}=\Phi(u_{j+t} \cdots u_{j-r})$~\footnote{Careful! Indices of
$\Z$ are decreasing from left to right.} for every $j$ in $\Z$.
\end{definition}

This means that the image of $u$ by $\varphi$ is obtained through
a  sliding window of length $p$. The parameter $r$ is called the
\emph{memory} and the parameter $t$ is called the
\emph{anticipation} of the function $\varphi$. We also write that
$\varphi$ is $(t,r)$-local. Such functions, restricted to finite
sequences, are computable by a parallel algorithm in constant
time.

\begin{definition}\label{digitsetconv}
Given a base $\beta$ with $|\beta| >1$ and two alphabets
$\mathcal{A}$ and $\mathcal{B}$ containing $0$, a {\em digit set
conversion} in base $\beta$ from $\mathcal{A}$ to $\mathcal{B}$ is
a function $\varphi: \mathcal{A}^\Z \rightarrow \mathcal{B}^\Z $
such that
\begin{enumerate}
    \item for any $u=(u_j)_{j \in \Z} \in \mathcal{A}^\mathbb{Z}$ with a finite number
     of non-zero digits, $v=(v_j)_{j\in \Z} = \varphi(u)\in \mathcal{B}^{\Z}$ has  only a finite number of non-zero digits, and
    \item $\sum\limits_{j\in \mathbb{Z}} v_j\beta^j = \sum\limits_{j\in \mathbb{Z}} u_j\beta^j$.
\end{enumerate}
Such a conversion is said to be {\em computable in parallel} if it
is a $p$-local function for some $p \in \mathbb{N}$.
\end{definition}

Thus, addition on ${\rm Fin}_{\mathcal{A}}(\beta)$ is computable
in parallel if there exists a digit set conversion in base $\beta$
from $\mathcal{A}+\mathcal{A}$ to $\mathcal{A}$ which is
computable in parallel.
We are interested in the following question:\\

\centerline{\it Given a base $\beta \in \mathbb{C}$, which
alphabet  $\mathcal{A}$ allows parallel addition on ${\rm
Fin}_{\mathcal{A}}(\beta)$ ?}

\medskip

If we restrict ourselves to integer alphabets $\mathcal{A} \subset
\mathbb{Z}$, then the necessary
condition~\eqref{closednessUnderAddition} implies that $\beta$ is
an algebraic number, i.e., $\beta$ is a zero of a non-zero
polynomial with integer coefficients. In \cite{FrPeSv}, we have
studied a more basic question: For which algebraic number $\beta$
does there exist at least one alphabet allowing parallel addition?
We have proved the following statement.

\begin{theorem}\label{alphabetExists}
Let $\beta$ be an algebraic number such that $|\beta |>1$ and all
its conjugates in modulus differ from 1. Then there exists an
alphabet $\mathcal{A}\subset \mathbb{Z}$ such that addition on
${\rm Fin}_{\mathcal{A}}(\beta)$ can be performed in parallel.
\end{theorem}

The proof of this theorem is constructive, the obtained alphabet
has the form of
a symmetric set of contiguous integers $\mathcal{A} = \{- a, -a+1, \ldots, -1, 0, 1, \ldots, a-1, a\}$ and, in general, $a$ is not minimum.\\

In this article, we address the question of minimality of the
alphabet allowing parallel addition. In the whole text we assume:
\begin{itemize}
    \item the base $\beta$ is an algebraic number such that $|\beta|>1$,
    \item the alphabet $\mathcal{A}$  is a finite set of consecutive integers containing $0$ and $1$, i.e., $\mathcal{A}$ is of the form
\begin{equation}\label{alphabet}
\mathcal{A} = \{ m,  m+1, \ldots, 0, 1, \ldots M -1 ,M\}\,, \quad
\text{where} \ m \leq 0 < M \ \hbox{and} \  m, M \in \mathbb{Z}.
\end{equation}
\end{itemize}

\begin{remark}\label{inverse}
Despite the usual requirement that a base $\beta$ is in  modulus
bigger than one, we can define  the set ${\rm
Fin}_{\mathcal{A}}(\beta)$ even  in the case  $|\beta|<1$ and ask
whether addition in this set can be performed in parallel. Since
for any $\beta\in \mathbb{C}\setminus\{0\}$, it holds
$${\rm Fin}_{\mathcal{A}}(\beta) =  {\rm Fin}_{\mathcal{A}}(\tfrac{1}{\beta}),$$
a $p$-local function performing parallel addition can be  found
either for both the sets ${\rm Fin}_{\mathcal{A}}(\beta)$ and
${\rm Fin}_{\mathcal{A}}(\tfrac{1}{\beta})$, or for none of them.
\end{remark}

\begin{remark}\label{conjugate}
Let $\beta$ and $\gamma$ be two different  algebraic numbers with
the same minimal polynomial and $\sigma: \mathbb{Q}(\beta)\mapsto
\mathbb{Q}(\gamma)$ be the isomorphism induced by $\sigma(\beta) =
\gamma$. If $\mathcal{A} \subset \mathbb{Z}$,  then $${\rm
Fin}_{\mathcal{A}}(\gamma) = \{\sigma(x)\,| \,  x\in {\rm
Fin}_{\mathcal{A}}({\beta})\}$$ and, for any integers $a_j$,
$b_j$, $c_j$, and for any finite coefficient sets $I_1, I_2
\subset \mathbb{Z}$,
$$\sum_{j \in I_1} (a_j + b_j) \beta^j = \sum_{j \in I_2}
 c_j \beta^j \quad \Longleftrightarrow \quad \sum_{j \in I_1}
 (a_j + b_j) \gamma^j = \sum_{j \in I_2} c_j\gamma^j.$$
Therefore, a $p$-local function performing parallel addition
 exists either simultaneously for both the sets
 ${\rm Fin}_{\mathcal{A}}(\gamma)$ and ${\rm Fin}_{\mathcal{A}}(\beta)$,
 or for none of them.
\end{remark}


\section{Lower bounds on the cardinality of alphabet allowing parallelism}\label{secLowerBound}

In this section, we give two lower bounds on the cardinality  of
alphabet $\mathcal{A}$ allowing parallel addition in the set ${\rm
Fin}_{\mathcal{A}}(\beta)$.

\begin{theorem}\label{zdola2}
Let $\beta$ be a positive real algebraic number, $\beta >1$,  and
let $\mathcal{A}$ be a finite set of contiguous integers
containing $0$ and $1$. If addition in ${\rm
Fin}_{\mathcal{A}}(\beta)$ can be performed in parallel, then
$\#\mathcal{A} \geq \lceil\beta\rceil$.
\end{theorem}

\begin{proof}
For any alphabet $\mathcal{B}$, denote
$$Z_{\mathcal{B}}=Z_{\mathcal{B}}(\beta):= \Bigl\{ \sum_{j=0}^n
s_j\beta^j\ |\ s_j \in \mathcal{B}, n \in \N\Bigr\}.$$ At first we
recall a result from \cite{ErKo}.  For an integer $q>0$, let
$\mathcal{Q}_q=\{0,1, \ldots, q\}$. Erd\"os and Komornik proved
the following: If $\beta \leq q+1$, then any closed interval
$[\alpha, \alpha +1]$ with $\alpha >0$ contains at least one point
from $Z_{\mathcal{Q}_q}$, i.e., $[\alpha, \alpha +1]\cap
Z_{\mathcal{Q}_q}\neq \emptyset$ for any $\alpha >0$.

We use the notation $m = \min \mathcal{A} \leq 0$ and $M = \max
\mathcal{A} \geq 1$.  Suppose, for contradiction, that $\#
\mathcal{A} = M-m+1 < \beta$. In particular, this assumption
implies that, for any $n \in \mathbb{N}$:
\begin{equation}\label{neigbour}
x_n := \beta^n + \sum_{j=0}^{n-1} m \beta^j > 0 \quad
\hbox{and}\quad y_n:= \sum_{j=0}^{n} M \beta^j < \beta^{n+1}.
\end{equation}
We can see that, for any $n \in \mathbb{N}$,  $y_n > x_n$, and,
additionally, since $x_n-y_{n-1} = \beta^n -
\sum_{j=0}^{n-1}(M-m)\beta^j >
\frac{\beta^{n}(\beta-M+m-1)}{\beta-1} >0$, we have $$ x_1 < y_1<
x_2< y_2 < x_3 < y_3 < x_4 < y_4 < \cdots$$ Consider an element
$x$ from $Z_{\mathcal{A}}=Z_{\mathcal{A}}(\beta)$.  It can be
written in the form $x = \sum_{j=0}^\ell a_j \beta^j$, with $a_j
\in \mathcal{A}$, where $a_\ell \neq 0$. If the leading
coefficient $a_\ell \leq -1$, then $x = \sum_{j=0}^\ell a_j
\beta^j \leq  -\beta^\ell + \sum_{j=0}^{\ell -1}M\beta^j$, and,
according to \eqref{neigbour}, the number $x$ is negative. It
means that any positive  element $x \in Z_{\mathcal{A}}$ can be
written as $x = \sum_{j=0}^\ell a_j \beta^j$, where $a_\ell \geq
1$, and, clearly,
$$x_\ell \leq x \leq y_\ell. $$
Thus, the intersection of $Z_{\mathcal{A}}$ with the  open
interval $(y_{n-1}, x_n )$ is empty for any $n \in \mathbb{N}$,
or, equivalently, $y_{n-1}$ and $x_n$ are the closest neighbors in
$Z_{\mathcal{A}}$. The gap between them is $x_n-y_{n-1}$, and it
tends to infinity with increasing $n$.

The existence of  a $p$-local function performing addition in
${\rm Fin}_{\mathcal{A}}(\beta)$ implies that, for any $x,y \in
Z_{\mathcal{A}}$, the sum $x+y$ has a $(\beta,
\mathcal{A})$-representation $x+y = \sum_{j=-p}^{n+p} z_j\beta^j$
with $z_j \in \mathcal{A}$, or, equivalently, $$Z_{\mathcal{A}} +
Z_{\mathcal{A}} \subset \frac{1}{\beta^p} Z_{\mathcal{A}}.$$ As $1
\in \mathcal{A}$, for any positive integer $q$ we obtain
\begin{equation}\label{gaps}
Z_{\mathcal{Q}_q} \subset \underbrace{Z_{\mathcal{A}}+   \cdots +
Z_{\mathcal{A}}}_{q\ {\rm times}} \subset
\frac{1}{\beta^{qp}}Z_{\mathcal{A}}
\end{equation}
Let us fix $q = \lfloor  \beta \rfloor$. Since $q+1 \geq \beta$,
then, according to the result of Erd\"os and Komornik, the gaps
between two consecutive elements in the set $Z_{\mathcal{Q}_q}$
are at most $1$. The set $\frac{1}{\beta^{qp}}Z_{\mathcal{A}}$ is
just a scaled copy of $Z_{\mathcal{A}}$ and thus
$\frac{1}{\beta^{qp}}Z_{\mathcal{A}}$ has arbitrary large gaps.
This contradicts  the inclusion~\eqref{gaps}.
\end{proof}

\begin{remark}
The inequality $ \#\mathcal{A}\geq \lceil \beta\rceil$  guarantees
that ${\rm Fin}_{\mathcal{A}}(\beta)$ is dense in $\mathbb{R}^+$
or in $\mathbb{R}$, depending on the fact whether the digits of
$\mathcal{A}$ are non-negative, or not. This property is very
important as it enables to approximate each positive real number
(resp. real number) by an element from ${\rm
Fin}_{\mathcal{A}}(\beta)$ with arbitrary accuracy.
\end{remark}

Using Remarks \ref{inverse} and \ref{conjugate} we can  weaken the
assumptions of Theorem \ref{zdola2}.

\begin{corollary}
Let $\beta$ be an algebraic number with at least one  positive
real conjugate (possibly $\beta$ itself) and let $\mathcal{A}$ be
an alphabet of contiguous integers containing $0$ and $1$. If
addition in ${\rm Fin}_{\mathcal{A}}(\beta)$ can be performed in
parallel, then
$$\#\mathcal{A} \geq \max \{ \lceil\gamma\rceil \,|\ \gamma \ \text{or} \  \gamma^{-1} \ \hbox{is a positive conjugate of }\ \beta\}.$$
\end{corollary}

When $\beta$ is an algebraic integer, and not only an algebraic
number, we can obtain another lower bound on the cardinality of
alphabet for parallelism:

\begin{theorem}\label{zdola}
Let $\beta$, with $|\beta| > 1$, be an algebraic integer of degree
$d$ with minimal polynomial $f(X) = X^d - a_{d-1}X^{d-1} -
a_{d-2}X^{d-2}- \cdots - a_1X-a_0$. Let $\mathcal{A}$ be an
alphabet of contiguous integers containing $0$ and $1$. If
addition in ${\rm Fin}_{\mathcal{A}}(\beta)$ is computable in
parallel, then $\# \mathcal{A} \geq |f(1)|$. If, moreover, $\beta$
is a positive real number, $\beta > 1$, then $\# \mathcal{A} \geq
|f(1)| +2$.
\end{theorem}

Firstly, we prove several auxiliary statements with  less
restrictive assumptions on the alphabet than required in
Theorem~\ref{zdola}.

In order to emphasize that the used alphabet is not  necessarily
in the form \eqref{alphabet}, we will denote it by $\mathcal{D}$.
We suppose that addition in ${\rm Fin}_{\mathcal{D}}(\beta)$ is
performable in parallel, which means that there exists a $p$-local
function $\varphi: (\mathcal{D}+\mathcal{D})^{\Z} \rightarrow
\mathcal{D}^{\Z}$ with memory $r$ and anticipation $t$, and
$p=r+t+1$, defined by the function $\Phi :
(\mathcal{D}+\mathcal{D})^p\rightarrow \mathcal{D}$, as introduced
in Definitions~\ref{local} and \ref{digitsetconv}. We work in the
set $\mathbb{Z}[\beta] = \{ b_0+b_1\beta + b_2\beta^2+ \cdots +
b_{d-1}\beta^{d-1}\ |\  b_j \in \mathbb{Z}\}$. Since $\beta$ is an
algebraic integer, the set $\mathbb{Z}[\beta]$ is a ring.

Let us point out that the following claim does not assume that the
digits are integers:

\begin{claim}\label{divides1}
Let $\beta$ be an algebraic number, and let $\mathcal{D}$ be  a
finite set such that $0 \in \mathcal{D} \subset
\mathbb{Z}[\beta]$. Then, for any $x \in \mathcal{D} +
\mathcal{D}$, the number $\Phi(x^p)-x$  belongs to the set
$(\beta-1) \mathbb{Z}[\beta]$.
\end{claim}

\begin{proof}
Let us denote $y: = \Phi(x^p)$. For any $n \in \mathbb{N}$, we
denote by $S_n$ the number represented by the string
\begin{equation}\label{number1}
\lexp{\omega}0 \underbrace{x\cdots x}_{t\ {\rm
times}}\,\underbrace{xxx\cdots xx x}_{n\ {\rm times}} \,\bullet
\,\underbrace{x\cdots x}_{r\ {\rm times}} 0^\omega.
\end{equation}
After the conversion by function $\Phi$, we obtain the  second
representation of the number $S_n$:
\begin{equation}\label{number2}
\lexp{\omega}0 w_{p-1}w_{p-2}\cdots w_2w_1\,\underbrace{yyy\cdots
yy y}_{n\ {\rm times}}\, \bullet\,
\widetilde{w}_{1}\widetilde{w}_2\cdots
\widetilde{w}_{p-1}0^\omega,
\end{equation}
where
\begin{equation}\label{number6}
w_j = \Phi(0^j x^{p-j}) \in \mathcal{D} \quad \hbox{ and }\quad
\widetilde{w}_j = \Phi(x^{p-j}0^j) \in \mathcal{D} \quad \hbox{
for} \quad j=1,2,\ldots, p-1\,.
\end{equation}
Put $W:= w_{p-1}\beta^{p-2}+ \cdots + w_2\beta + w_1$ and
 $\widetilde{W}:= \widetilde{w}_{1}\beta^{p-2}+ \cdots + \widetilde{w}_{p-2}\beta + \widetilde{w}_{p-1}$. Let
 us stress that neither $W$ nor $\widetilde{W}$ depend on $n$. Comparing both the representations (\ref{number1}) and (\ref{number2})
 of the number $S_n$, we obtain
$$ S_n =x \sum\limits_{j=-r}^{n+t-1}\beta^j = W \beta^n + y \sum\limits_{j=0}^{n-1}\beta^j + \widetilde{W}\beta ^{-p+1},$$
i.e.,
\begin{equation}\label{number4}
x \frac{\beta^{n+t}-1}{\beta-1} + x \sum\limits_{j=-r}^{-1}\beta^j
= W\beta^n + y\frac{\beta^{n}-1}{\beta-1} + \widetilde{W}
\beta^{-p+1} \quad \hbox{for any } \ n \in \mathbb{N}.
\end{equation}
Subtracting these equalities (\ref{number4}) for $n=\ell+1$ and
$n=\ell$, we get
\begin{equation}\label{number3}
x\beta^{\ell+t} = W\beta^{\ell+1} - W \beta^{\ell} + y\beta
^{\ell} \quad \Longrightarrow \quad x(\beta^t - 1) = W(\beta-1) +
y-x.
\end{equation}
Since $\beta^t - 1=(\beta-1)(\beta^{t-1} + \cdots+ \beta+1)$, the
number $y-x$ can be expressed in the form $(\beta-1)
\sum_{k=0}^mw'_k\beta^k$ with $w'_k \in \mathbb{Z}$.
\end{proof}

A technical detail concerning the value of $W$ in the course of
the previous proof  (Equation \eqref{number3}) will be important
in the sequel as well.  Let us point out this detail.

\begin{corollary}\label{doubleW} Let $\beta$ be an algebraic number,  $\mathcal{E}\subset \mathbb{Z}[\beta]$
and $\mathcal{D}\subset \mathbb{Z}[\beta]$ be two alphabets
containing $0$.  Suppose that there exists a $p$-local digit set
conversion $\xi: \mathcal{E}^\Z \rightarrow \mathcal{D}^\Z$
defined by the function $\Xi: \mathcal{E}^p \rightarrow
\mathcal{D}$, $p=r+t+1$. Then
$$\sum_{j=1}^{p-1}\Xi(0^j
x^{p-j})\beta^{j-1} = \frac{x\beta^t - \Xi(x^p)}{\beta -1}\qquad
\hbox{for any } \ \ x \in \mathcal{E}\,.$$
\end{corollary}

\begin{claim}\label{divides2}
Let $\beta$ be an algebraic integer and let $\mathcal{D}$ be a
finite set of (not necessarily contiguous) integers containing
$0$. Then
$$\Phi(x^p) \equiv x \mod |f(1)| \quad \hbox{  for any } \ x \in \mathcal{D}+\mathcal{D}.$$
\end{claim}

\begin{proof}
According to Claim \ref{divides1}, the number $\beta -1$ divides
the integer $\Phi({x^p}) - x = y-x$ in the ring
$\mathbb{Z}[\beta]$, i.e.,
$$ y-x = (\beta -1)(c_0+c_1\beta + \cdots + c_{d-1}\beta^{d-1}) \quad \hbox{for some }  \ \ c_0, c_1, \ldots, c_{d-1} \in \mathbb{Z}.$$
As $\beta^d =  a_{d-1}\beta^{d-1} + a_{d-2}\beta^{d-2}+ \cdots
+a_1\beta+a_0$ and powers $\beta^0, \beta^1, \beta^2, \ldots,
\beta^{d-1}$ are linearly independent over $\mathbb{Q}$, we deduce
\begin{eqnarray*}
y-x & = & -c_0 + c_{d-1}a_{0}\\
0 & = & c_0 - c_1 +c_{d-1}a_{1}\\
0 & = & c_1 -c_2 + c_{d-1}a_{2}\\
  & & \vdots\\
0 & = & c_{d-3} -c_{d-2}+ c_{d-1}a_{d-2}\\
0 & = & c_{d-2} -c_{d-1}+ c_{d-1}a_{d-1}\\
\end{eqnarray*}
Summing up all these equations, we obtain $$ y-x = -c_{d-1}(1-
a_{0}- a_{1} \cdots - a_{d-1})= -c_{d-1} f(1),$$ which implies
Claim~\ref{divides2}.
\end{proof}

The following claim again allows a more general alphabet, but the
base must be a positive real number.

\begin{claim}\label{divides3}
Let $\beta$ be a real algebraic number, $\beta >1$, and let
$\mathcal{D}$ be a finite set, such that $ 0\in \mathcal{D}
\subset \mathbb{Z}[\beta]$. Denote $\lambda = \min \mathcal{D}$
and $\Lambda = \max \mathcal{D}$. Then $\Phi(\Lambda^p) \neq
\lambda$ and $\Phi(\lambda^p) \neq \Lambda$.
\end{claim}

\begin{proof}
Firstly, let us assume that $\Phi(\Lambda^p) = \lambda$. Put
$x=\Lambda$ and $y=\lambda$ into \eqref{number4} and use
\eqref{number3} for determining $W$. We get
$$ \Lambda \frac{\beta^{n+t}-1}{\beta-1} + \Lambda \sum\limits_{j=-r}^{-1}\beta^j = \left( \Lambda \frac{\beta^{t}}{\beta-1} - \lambda \frac{1}{\beta-1} \right)\beta^n + \lambda\frac{\beta^{n}-1}{\beta-1} + \widetilde{W} \beta^{-p+1} .$$
After cancellation of the same terms on both sides, we have to
realize that $\frac{1}{\beta-1} = \sum
_{j=1}^{\infty}\frac{1}{\beta^j}$, all digits in $\widetilde{W}$
are at least $\lambda$, and our base $\beta >1$. Therefore, we
obtain
$$ -\Lambda \sum_{j=r+1}^{\infty} \frac{1}{\beta^j} = -\lambda \sum_{j=1}^{\infty}\frac{1}{\beta^j} + \sum _{j=1}^{p-1}\frac{\widetilde{w}_j}{\beta^j} \geq - \lambda \sum_{j=p}^{\infty}\frac{1}{\beta^j},$$
which yields a contradiction, as the left side is negative, but the right one is non-negative.\\
The proof of $\Phi(\lambda^p) \neq \Lambda$ is analogous.
\end{proof}

\begin{claim}\label{divides4}
Let $\beta$ be a real algebraic number, $\beta >1$, and let
$\mathcal{D}$ be a finite set, such that $ 0\in \mathcal{D}
\subset \mathbb{Z}[\beta]$. Denote $\lambda = \min \mathcal{D}$
and $\Lambda = \max \mathcal{D}$. Then $\Phi(\Lambda^p) \neq
\Lambda$. If, moreover, $\lambda \neq 0$ then $\Phi(\lambda^p)
\neq \lambda$.
\end{claim}

\begin{proof}
We prove the claim by contradiction. Let us assume
$\Phi(\Lambda^p) = \Lambda$. For any $q\in \mathbb{N}$, denote
$T_q$ the number represented by
\begin{equation}\label{number7}
\lexp{\omega}0 \underbrace{\Lambda\cdots \Lambda}_{t\ {\rm
times}}\,\bullet \,\underbrace{\Lambda\cdots \Lambda}_{r\ {\rm
times}}\,\underbrace{(2\Lambda)(2\Lambda)\cdots
(2\Lambda)(2\Lambda)}_{q\ {\rm times}} \, 0^\omega.
\end{equation}
After conversion by the function $\Phi$, we get
\begin{equation}\label{number8}
\lexp{\omega}0 w_{p-1}w_{p-2}\cdots w_2w_1\,\bullet \,
z_1z_2\cdots z_{r+t+q} \, 0^\omega,
\end{equation}
where $w_j$ are defined in \eqref{number6} for $x=\Lambda$ and
$z_j \in \mathcal{D}$. The value $W=
\sum_{j=1}^{p-1}w_{j}\beta^{j-1}$ computed by \eqref{number3} is
now $W = \Lambda \frac{\beta^t-1}{\beta-1} = \Lambda
\sum_{j=0}^{t-1}\beta^j$. Using the representations
\eqref{number7} and \eqref{number8} for evaluation of the number
$T_q$,  and the fact that $z_j \leq \Lambda$ for any $j$, we
obtain
$$\Lambda \sum_{j=-r}^{t-1}\beta^j + (2\Lambda)\sum_{j=-r-q}^{-r-1}\beta^j =
 W + \sum_{j=1}^{r+t+q}z_j\beta^{-j} = \Lambda \sum_{j=0}^{t-1}\beta^j +
 \sum_{j=1}^{r+t+q}z_j\beta^{-j},$$
and thus
$$\Lambda \sum_{j=-r}^{-1}\beta^j +(2\Lambda)\sum_{j=-r-q}^{-r-1}\beta^j
\leq \Lambda \sum_{j=1}^{\infty}\beta^{-j} \quad \Longrightarrow \quad
\sum_{j=-r-q}^{-r-1}\beta^j \leq \sum_{j=q+r+1}^{\infty}\beta^{-j}.$$
Summing up  both sides of the last inequality, we get $\frac{1}{\beta^{q+r}}\,
\frac{\beta^q-1}{\beta-1}\leq \frac{1}{\beta^{q+r}}\,\frac{1}{\beta-1}$ for
all $q\in \mathbb{N}$, thus a contradiction.\\
The proof of $\Phi(\lambda^p) \neq \lambda$ is analogous.
\end{proof}

Now we can easily deduce the statement of Theorem \ref{zdola}:

\begin{proof}
Let $\mathcal{A} = \{m, m+1, \ldots, M-1, M\}$ be a set of
contiguous integers containing $0$ and $1$, i.e., $m \leq 0 <
M$.

Firstly, consider the base $\beta$ as any algebraic integer of
modulo greater than 1. If $|f(1)| =1$, there is nothing to prove.
Therefore, suppose now that $|f(1)| \geq 2$. Since $M+1 \in
\mathcal{A} + \mathcal{A}$, then, according to
Claim~\ref{divides2}, the digit $\Phi((M+1)^p)\leq M$ is congruent
to $M+1$ modulo $|f(1)|$. Therefore, necessarily, $M+1 - |f(1)|
\geq \Phi((M+1)^p) \geq m$. This implies the claimed inequality
$\# \mathcal{A} = M-m+1 \geq |f(1)|$.

Now suppose that $\beta>1$. According to Claims~\ref{divides3} and
\ref{divides4}, the digits $M$, $m$, and $\Phi(M^p)$ are distinct,
i.e., the alphabet $\mathcal{A}$ contains at least three
elements. Therefore, for the proof of $\# \mathcal{A}\geq |f(1)| +
2$, we can restrict ourselves to the case $|f(1)|\geq 2$. As $M
>\Phi(M^p) > m$ and $\Phi(M^p) \equiv M \mod |f(1)|$, we have $M -
|f(1)|\geq \Phi(M^p)\geq m+1$. It implies the second part of the
claim, namely that $\#\mathcal{A} = M-m+1 \geq |f(1)| + 2$.
\end{proof}

\bigskip
The assumptions of the previous Claims~\ref{divides2},
\ref{divides3}, and \ref{divides4} are much more relaxed than the
assumptions of Theorem~\ref{zdola}. Therefore, modified statements
can be proved as well. For instance, the following result holds.

\begin{proposition}
Given $\beta>1 $ an algebraic integer with minimal polynomial
$f(X)$, let $\mathcal{D}$ be a finite set of (not necessarily
contiguous) integers containing $0$, such that $\gcd \mathcal{D}
=1$ and $\min \mathcal{D} <0 < \max \mathcal{D}$. If addition in
${\rm Fin}_{\mathcal{D}}(\beta)$ is computable in parallel, then
$\#\mathcal{D} \geq |f(1)|+2$.
\end{proposition}

\begin{remark}
Exploiting Remarks~\ref{inverse} and \ref{conjugate}, we may also
strengthen Theorem~\ref{zdola}.
\begin{enumerate}
    \item If a polynomial $f(X)\in \mathbb{Z}[X]$ of degree $d$
     is the minimal polynomial of $\beta$, then $g(X)=X^df(\frac{1}{X})$
     is the minimal polynomial of $\frac{1}{\beta}$, and, moreover, $f(1) = g(1)$.
     Therefore, the assumption ``$\beta$ is an algebraic integer" in
     Theorem \ref{zdola} can be replaced by ``$\beta$ or $\frac{1}{\beta}$ is an algebraic integer".
    \item Even the second part of Theorem \ref{zdola} can be applied
    to a broader class of bases. The lower bound $\#\mathcal{A}
    \geq |f(1)|+2$ remains valid even if $\beta$ is an algebraic
    integer and one of its conjugates is a positive real number greater than 1.
\end{enumerate}
\end{remark}


\section{Addition versus subtraction and conversion}\label{sec:conversion}

As we have already mentioned, addition in the set ${\rm
Fin}_{\mathcal{A}}(\beta)$ can be interpreted as a digit set
conversion from alphabet $\mathcal{A}+\mathcal{A}$ into alphabet
$\mathcal{A}$. Let us point out that, if addition of two numbers
can be performed in parallel, then addition of three numbers can
be done in parallel as well, and the same holds for any fixed
number of summands. This implies that, if $\{-1,0,1\} \subset
\mathcal{A}$, then subtraction of two numbers from ${\rm
Fin}_{\mathcal{A}}(\beta)$ can be viewed as addition of fixed
numbers of summands, and therefore, no special study of
parallelism for subtraction of $(\beta,
\mathcal{A})$-representations is needed.

On the other hand, if the elements of $ \mathcal{A}$ are
non-negative and the base $\beta$ is a real number greater than
$1$, then the set ${\rm Fin}_{\mathcal{A}}(\beta) \subset [0,
+\infty)$ is not closed under subtraction. We may investigate only
the existence of a parallel algorithm for subtraction $y-x$ for
$y\geq x$. But even if $ {\rm Fin}_{\mathcal{A}}(\beta)$ is closed
under subtraction of $y-x$ for $y\geq x$, it is not possible to
find any parallel algorithm for it. Let us explain why: Suppose
that subtraction is a $p$-local function $\varphi$. Then $\varphi$
must convert a string with a finite number of non-zero digits into
a string with a finite number of non-zero digits. It forces the
function $\Phi$ associated with $\varphi$ (see
Definition~\ref{local}) to satisfy $\Phi(0^p)= 0$. Therefore, the
algorithm has no chance to exploit the fact that $y \ge x$, when
the $(\beta, \mathcal{A})$-representation of $y$ is
$\lexp{\omega}0 1 0^n\bullet 0^\omega$
and the $(\beta, \mathcal{A})$-representation of $x$ is $\lexp{\omega}0 1\bullet 0^\omega$.\\

Therefore, we are going to focus only on addition of $(\beta,
\mathcal{A})$-representations. We start with setting some
terminology:

\begin{definition}\label{algo}
Let $\beta$ with $|\beta| >1$ be fixed, and consider $c$ and $K$
from $\Z$, $K\ge 2$. The parameters $c$ and $K$ must be such that
$0$ is always an element of the considered alphabets (both before
and after the conversion).
\begin{itemize}
  \item {\em Smallest digit elimination (SDE)} in base $\beta$ is a digit set
  conversion from $\{c, \ldots, c+K\}$ to $\{c+1, \ldots, c+K\}$.
  \item {\em Greatest digit elimination (GDE)} in base $\beta$
  is a digit set conversion from $\{c, \ldots, c+K\}$ to $\{c, \ldots, c+K-1\}$.
\end{itemize}
\end{definition}

The following result  enables  to replace the alphabet
$\mathcal{A}+\mathcal{A}$ entering into conversion during parallel
addition by a smaller one. When  looking for parallel algorithms
for addition on  minimal alphabets, we will  separately discuss
the case when an alphabet contains only non-negative digits.

\begin{proposition}\label{conversion}
Let $\mathcal{A} = \{ m, m+1, \ldots, M-1, M\} $ be an alphabet of
contiguous integers containing $0$ and $1$ and let $\beta$ be the
base of the respective numeration system.

\begin{enumerate}
    \item If $m=0$, then addition in ${\rm Fin}_{\mathcal{A}}(\beta) $ can be
     performed in parallel if, and only if, the conversion from
     $ \mathcal{A} \cup \{ M+1\}$ into $\mathcal{A}$ (greatest digit elimination) can be performed in parallel.
    \item\label{both sides} Suppose that $\{-1,0,1\} \subset \mathcal{A}$. Then addition in ${\rm Fin}_{\mathcal{A}}(\beta) $ can be
     performed in parallel if, and only if, the conversion from    $ \mathcal{A} \cup \{ M+1\}$ into $\mathcal{A}$
   (greatest digit elimination) and
the conversion from
    $\{m-1\} \cup \mathcal{A}$ into $\mathcal{A}$ (smallest  digit elimination)
      can be performed in parallel.

\end{enumerate}
\end{proposition}

\begin{proof}
1. Consider $x$ and $y$ from ${\rm Fin}_{\mathcal{A}}(\beta)$, and
let $z= x+y$. The coefficients of $z$ are in $\{0, \ldots, 2M\}$,
so $z$ can be decomposed into the sum of $z'$ with coefficients in
$\{0, \ldots, M+1\}$ and $z''$ with coefficients in $\{0, \ldots,
M-1\}$. According to the assumption of Statement 1, $z'$ is
transformable in parallel into $w$ with coefficients in
$\mathcal{A}$. So $w+z''$ has coefficients in $\{0, \ldots,
2M-1\}$. We iterate this process until the result is on
$\mathcal{A}$, knowing that we have to repeat
$M$ such iterations (i.e., a finite fixed number of iterations).\\
2. Analogous to the proof of Statement 1; and, again, the number
of such iterations is finite and fixed, this time at $\max \{M,
-m\}$.

\end{proof}

In the sequel we will discuss only  questions about  parallel
addition on ${\rm Fin}_{\mathcal{A}}(\beta) $. Nevertheless,
parallel addition is closely related to the question of parallel
conversion between different alphabets.

\begin{corollary} Let   $ \mathcal{A} $  and   $ \mathcal{B}$ be two alphabets
of consecutive integers containing 0.

\begin{enumerate}
      \item Suppose that $\{-1,0,1\} \subset \mathcal{A}$ and
      addition on ${\rm Fin}_{\mathcal{A}}(\beta) $ can be
      performed in parallel. Then conversion from $\mathcal{B}$
      into $\mathcal{A}$ can be performed in parallel for any
      alphabet $\mathcal{B}$.
   \item Suppose that conversion from $\mathcal{B}$ to
   $\mathcal{A}$ and conversion from $\mathcal{A}$ to
   $\mathcal{B}$ can be performed in parallel. Then parallel
   addition on ${\rm Fin}_{\mathcal{A}}(\beta) $ can be
      performed in parallel if, and only if, parallel
   addition on  ${\rm Fin}_{\mathcal{B}}(\beta) $ can be
      performed in parallel.
\end{enumerate}

\end{corollary}

\begin{proof} 1. Possibility of parallel
   addition on ${\rm Fin}_{\mathcal{A}}(\beta) $ implies that
conversion

$$\hbox{from} \ \ \underbrace{\mathcal{A} +
\mathcal{A} + \cdots
   +\mathcal{A}}_{k\ \ \textrm{times} }\quad \hbox{ into}\quad  \mathcal{A}$$ can be made in
   parallel for any fixed positive integer $k$. Any
   finite alphabet  $\mathcal{B}$ is a subset of $\underbrace{\mathcal{A} + \mathcal{A} + \cdots
   +\mathcal{A}}_{k\ \ \textrm{times}}$ for some $k$. This proves Point
   1.\\
   2.  Let us assume that parallel addition  is possible on
${\rm Fin}_{\mathcal{A}}(\beta) $. To add two numbers
   $x$ and $y$  represented on  the alphabet $\mathcal{B}$, we at
   first use parallel algorithm for conversion  from $\mathcal{B}$  to
   $\mathcal{A}$, then we add these numbers by parallel algorithm
   acting on ${\rm Fin}_{\mathcal{A}}(\beta) $ and finally we use
   parallel
   algorithm for conversion back   from $\mathcal{A}$  to
   $\mathcal{B}$.
\end{proof}

We now show how a parallel algorithm acting on one alphabet can be
modified to work on another alphabet. First we mention a simple
property.

\begin{proposition}\label{opposite}
Given a base $\beta \in \mathbb{C}$, $\beta$ an algebraic number,
and two alphabets $\mathcal{A}$ and $\mathcal{B}$ containing $0$
such that $ \mathcal{A} \cup \mathcal{B} \subset
\mathbb{Z}[\beta]$. Then conversion in base $\beta$ from
$\mathcal{A}$ to $\mathcal{B}$ is computable in parallel by a
$p$-local function if, and only if, conversion in base $\beta$
from $(-\mathcal{A})$ to $(-\mathcal{B})$ is computable in
parallel by a $p$-local function.
\end{proposition}
\begin{proof}
Let $\varphi: \mathcal{A}^\mathbb{Z} \rightarrow
\mathcal{B}^\mathbb{Z} $ be $p$-local, defined by $\Phi:
\mathcal{A}^p \rightarrow \mathcal{B} $. Conversion from the
alphabet $(-\mathcal{A}) =\{-a\,|\, a\in \mathcal{A}\}$ to
$(-\mathcal{B})$ is computable in  parallel by the $p$-local
function $\tilde{\varphi}: (-\mathcal{A})^\mathbb{Z} \rightarrow
(-\mathcal{B})^\mathbb{Z}$ which uses  the function $\tilde{\Phi}:
(-\mathcal{A})^p \rightarrow (-\mathcal{B})$ defined for any
$x_1,x_2, \ldots, x_p \in (-\mathcal{A})$ by the prescription
$$\tilde{\Phi}(x_1x_2 \cdots x_p) = -\Phi\bigl((-x_1) (-x_1)\cdots
(-x_p)\bigr)\,,$$ which implies that
$\tilde{\Phi}(0^p)=-\Phi(0^p)=0$.
\end{proof}

The next result allows to pass from one alphabet allowing parallel
digit set conversion to another one. First we set a definition.

\begin{definition}\label{fixed}
Let $\mathcal{A}$ and $\mathcal{B}$ be two alphabets containing
$0$ such that $ \mathcal{A} \cup \mathcal{B} \subset
\mathbb{Z}[\beta]$. Let $\varphi: \mathcal{A}^\mathbb{Z}
\rightarrow \mathcal{B}^\mathbb{Z} $ be a $p$-local function
realized by the function $\Phi: \mathcal{A}^p \rightarrow
\mathcal{B}$. The letter $h$ in $\mathcal{A}$ is said to be
\emph{fixed} by $\varphi$ if $\varphi(\lexp{\omega}h \bullet
h^\omega ) = \lexp{\omega}h \bullet h^\omega $, or, equivalently,
$\Phi(h^p)=h$.
\end{definition}

\begin{theorem}\label{symmetry}
Given a base $\beta \in \mathbb{C}$, $\beta$ an algebraic number,
and two alphabets $\mathcal{A}$ and $\mathcal{B}$ containing $0$
such that $ \mathcal{A} \cup \mathcal{B} \subset
\mathbb{Z}[\beta]$, suppose that conversion in base $\beta$ from
$\mathcal{A}$ to $\mathcal{B}$ is computable by a $p$-local
function $\varphi: \mathcal{A}^\mathbb{Z} \rightarrow
\mathcal{B}^\mathbb{Z} $.

If some letter $h$ in $\mathcal{A}$ is fixed by $\varphi$ then
conversion in base $\beta$ from $ \mathcal{A}'=\{a-h \,|\, a\in
\mathcal{A}\}$ to $\mathcal{B}'=\{b-h \,|\, b  \in \mathcal{B}\}$
is computable in parallel by a $p$-local function.
\end{theorem}

\begin{proof}   Let $\Phi: \mathcal{A}^p \rightarrow \mathcal{B} $  be
the function realizing  conversion  from $\mathcal{A}$ to
$\mathcal{B}$, with memory $r$ and anticipation $t$ satisfying
$p=r+t+1$. It means that for any $u =(u_j) \in
\mathcal{A}^\mathbb{Z} $ such that $u$ has only finite number of
non-zero entries, we have after conversion  the sequence $v
=\varphi(u)$ such that
\begin{itemize}
\item $v= (v_j) \in \mathcal{B}^\mathbb{Z}$ has only finite number
of non-zero entries;

\item $v_j = \Phi(u_{j+t} \cdots u_{j+1} u_j u_{j-1}\cdots
u_{j-r})$ for any  $j \in \mathbb{Z}$;

\item $\sum_{j\in \mathbb{Z}} u_j\beta^j =\sum_{j\in \mathbb{Z}}
v_j\beta^j$.

\end{itemize}
For any $x_1,\dots,x_p \in \mathcal{A'}$ we define
\begin{equation}\label{DefPsi} \Psi(x_1x_2 \cdots x_p) =
\Phi\bigl((x_1+h)(x_2+h)\cdots (x_p+h)\bigr) - h\,.\end{equation}
It is easy to check that $\Psi: (\mathcal{A}')^p \rightarrow
\mathcal{B}'$. Denote by $\psi: (\mathcal{A}')^\mathbb{Z}
\rightarrow (\mathcal{B}')^\mathbb{Z}$ the $p$-local function
realized by the function $\Psi$. We will show that the function
$\psi$ performes conversion from $\mathcal{A}'$ to $\mathcal{B}'$.

As $\Phi(h^p) = h$ we have $\Psi(0^p) = \Phi(h^p)-h =0$.
Consequently,  $v'=\psi(u')$  has only a finite numbers of
non-zero digits of the form
$$v_j' = \Psi(u_{j+t}' \cdots
u'_{j}\cdots u'_{j-r})$$ for any $u' \in
(\mathcal{A}')^\mathbb{Z}$ with a finite number of non-zero
entries $u_j'$. It remains to show that
\begin{equation}\label{toVerify}\sum_{j\in \mathbb{Z}}
u_j'\beta^j =\sum_{j\in \mathbb{Z}} v_j'\beta^j =  \sum_{j\in
\mathbb{Z}} \Psi(u'_{j+t} \cdots u'_j\cdots
u'_{j-r})\beta^j\,.\end{equation} Before verifying the previous
statement,  we deduce an auxiliary equality. Put $L:= \max\{j \in
\mathbb{Z}\,|\, u'_j \neq 0\}$ and define $u=(u_j) \in
\mathcal{A}^\mathbb{Z}$ as $$ u_j:= \left\{
\begin {array}{cl}
u'_j+h & \hbox{\ if\ \ }\ j\leq L\\
h  & \hbox{\ if\ \ }\  L<  j\leq L+p-1\\
0&  \hbox{\ if\ \ }\ j\geq L+p
\end{array}  \right.$$

As $\varphi$ realizes conversion from $\mathcal{A}$ to
$\mathcal{B}$,  we have
\begin{equation}\label{convert}  h \!\!\!\sum_{j\leq L+p-1}\beta^j
+ \sum_{j\leq L}u'_j\beta^j = \sum_{j\in \mathbb{Z}} u_j\beta^j  =
\sum_{j \in \mathbb{Z}} \Phi(u_{j+t} \cdots u_j\cdots
u_{j-r})\beta^j = \sum_{j\in \mathbb{Z}} v_j\beta^j\,.
\end{equation}
Let us split the last sum into three pieces
$$ P_1 = \sum_{j\geq L+p+r} v_j\beta^j, \qquad P_2 =
\sum_{j=L+r+1}^{L+p+r-1} v_j \beta^j \quad \hbox{and} \quad P_3 =
\sum_{j\leq L+r} v_j\beta^j\,.$$ In the first sum,
$v_j=\Phi(0^p)=0$,  as for $j\geq L+p+r$, all arguments $u_{j+t},
\ldots , u_j,\ldots ,u_{j-r}$ of the function $\Phi$ are zeros,
i.e., $ P_1=0$.

In the second sum $P_2$, the first coefficient is
$v_{L+r+1}=\Phi(u_{L+p} \cdots u_{L+1})= \Phi(0h^{p-1})$, the
second one is $v_{L+r+2}=\Phi(u_{L+p+1} \cdots u_{L+2}) =
\Phi(0^2h^{p-2})$, etc. Using Corollary \ref{doubleW},  we obtain
$$P_2 = \beta^{L+r+1}\sum_{j=1}^{p-1}\Phi(0^jh^{p-j})\beta^{j-1} =  \beta^{L+r+1} h \frac{\beta^t-1}{\beta-1}.$$
 Since
$\sum_{j\in \mathbb{Z}} v_j\beta^j = P_1+P_2+P_3$, we may
calculate the value of $P_3$ using \eqref{convert}
\begin{equation}\label{P3}
P_3= \sum_{j\leq L} u_j'\beta^j +h \sum_{j\leq L+p-1}\beta^j -
\beta^{L+r+1} h \frac{\beta^t-1}{\beta-1} = \sum_{j\leq
L}u_j'\beta^j + h \sum_{j\leq L+r}\beta^j \,.
\end{equation}
All coefficients $v_j's$ in the sum $P_3$ are of the form
$v_j=\Phi\bigl((u_{j+t}'+h) \cdots (u_{j-r}'+h)\bigr)$. We have
thus shown that
\begin{equation}\label{P3Psi}
\sum_{j\leq L+r} \Phi\bigl((u_{j+t}'+h) \cdots
(u_{j-r}'+h)\bigr)\beta^j= \sum_{j\leq L}u_j'\beta^j +h
\sum_{j\leq L+r}\beta^j\,.
\end{equation}
Let us come back to the task to show \eqref{toVerify}. In the
right sum of \eqref{toVerify}, all arguments  $u'_{j+t}, \ldots,
u'_j,\ldots ,u'_{j-r}$  of $\Psi$  are zero for $j> L+r$, and
therefore $v_j'=\Psi(0^p) = \Phi(h^p)-h=0$. In the left sum of
\eqref{toVerify}, all coefficients $u_j'$ are for $j
>L$ equal to zero as well. So we have to check whether
$$ \sum_{j\leq L}
u_j'\beta^j =  \sum_{j\leq L+r} \Psi(u'_{j+t} \cdots u'_j\cdots
u'_{j-r})\beta^j\,.
$$
Because of the definition of $\Psi$ in \eqref{DefPsi},  this
relation is equivalent to Equation \eqref{P3Psi}.
\end{proof}

\begin{remark}  For deduction of \eqref{convert}, we have applied
the mapping $\varphi$ to the word $u
=\lexp{\omega}0u_{L+p-1}u_{L+p-2}\cdots u_0\bullet u_{-1}u_{-2}
\cdots $ with infinitely many non-zero entries. Let us explain the
correctness of this step. Denote by $u^{(n)}$ the word
$\lexp{\omega}0u_{L+p-1}u_{L+p-2}\cdots u_0\bullet u_{-1} \cdots
u_{-n} 0^\omega$. Since $u^{(n)}$ has only a finite number of
non-zero digits, we know that the value corresponding to
$\varphi(u^{(n)})$ equals the value corresponding to $v^{(n)} =
\varphi(u^{(n)})$.  Clearly  $u_{n} \to u$ and $\varphi(u^{(n)})
\to \varphi(u)$ as $n \to \infty$ in the product topology. The
same is true for the numerical values represented by these words.
\end{remark}

\medskip

In the following sections, we give parallel algorithms  for
addition in a given base on alphabets (of contiguous integers)
containing~$0$, of the minimal cardinality~$K$. While doing so, we
favour the method of starting with an alphabet containing only
non-negative digits, and writing a parallel algorithm for the
greatest digit elimination, Algorithm~\textsl{GDE}($\beta$),
converting representations on $\{0, 1\ldots, K-1,K\}$ into
representations on $\{0, 1\ldots, K-1\}$. By
Proposition~\ref{conversion}, parallel addition is thus possible
on $\{0, 1\ldots, K-1\}$. In order to show that parallel addition
is possible also on other alphabets (of the same size), we use the
following corollary.

\begin{corollary}\label{PossibleShifts}  For $K, d \in \mathbb{Z}$, where $ 0\leq d\leq K-1$, denote
 $$\mathcal{A}_{-d}=\{-d, \ldots,0,\ldots, K-1-d \}\,. $$
Let $\varphi$ be a $p$-local function realizing conversion in base
$\beta$ from $\mathcal{A}_{0}\cup \{K\}$ to $\mathcal{A}_{0}$. If
both letters $d$ and $K-1-d$ are fixed by $\varphi$, then parallel
addition is performable in parallel on $\mathcal{A}_{-d}$ as well.
\end{corollary}
\begin{proof}
According to Theorem~\ref{symmetry}, conversions from $\{-d,
\ldots,0,\ldots, K-1-d, K-d\}$ into $\{-d, \ldots,0,\ldots, K-1-d
\} $ and also from $\{-K+1+d, \ldots,0,\ldots, d+1 \}$ into
$\{-K-1+d, \ldots,0,\ldots, d \} $  are performable in parallel.
According to Proposition~\ref{opposite}, conversion from $\{-d-1,
 \ldots,0,\ldots, K-1-d \}$ into $\{-d, \ldots,0,\ldots, K-1-d
\} $ is performable in parallel, as well. Using
Proposition~\ref{conversion} Point~\eqref{both sides}, addition on
the alphabet $\mathcal{A}_{-d}$ can be made in parallel.
\end{proof}
\medskip

\section{Integer base and related complex numeration systems}\label{known}

In this section, we consider some well studied numeration systems,
where the base is an integer, or a root of an integer. Parallel
algorithms for addition in these systems can be found in
\cite{Fr}, but the question of minimality of the alphabet was not
discussed there.

\subsection{Positive integer base}\label{positive}

If the base  $\beta$ is a positive integer $b\ge 2$, then the
minimal polynomial is $f(X) = X-b$, and Theorem~\ref{zdola} gives
$\#\mathcal{A} \geq |f(1)|+2 = b+1$. It is known that parallel
addition is feasible on any alphabet of cardinality $b +1$
containing $0$, in particular on alphabets $\mathcal{A} = \{0,1,
\ldots, b\}$ and  $\mathcal{A} = \{-1, 0,1,\ldots, b-1\}$, see for
instance Parhami~\cite{Parhami}. In the case that $b$ is even,
$b=2a$, parallel addition is realizable on the alphabet
$\mathcal{A} = \{-a, \ldots, a\}$ of cardinality $b +1$ by the
algorithm of Chow and Robertson~\cite{ChowRobertson}.


\subsection{Negative integer base}\label{negative}

If the base $\beta$ is a negative integer, $\beta =-b$, $b\ge 2$,
then the minimal polynomial is $f(X) = X+b$, and
Theorem~\ref{zdola} gives the bound $\#\mathcal{A} \geq |f(1)| =
b+1$. In this section we prove

\begin{theorem}\label{neg-par0} Let  $\beta=-b\in \mathbb{Z}$, $b \ge 2$.
Any alphabet $\mathcal{A}$ of contiguous integers containing $0$
with cardinality  $\# \mathcal{A} = b+1$ allows parallel addition
in base $\beta=-b$ and this alphabet cannot be further reduced.
\end{theorem}

Any alphabet of contiguous integers containing $0$ which has
cardinality  $ b+1$ can be written in the form
$$\mathcal{A}_{-d}=\{-d, \ldots,0,\ldots, b-d \}\quad \hbox{ for}
\  \ 0 \leq d \leq b-1\,.$$ For proving Theorem \ref{neg-par0}, we
firstly consider the alphabet consisting only of non-negative
digits, i.e., the alphabet $\mathcal{A}_0$.

\vskip0.3cm \hrule \vskip0.2cm

\noindent {\bf Algorithm~\textsl{GDE}($-b$)}: Base $\beta= -b$, $b
\ge 2$, parallel conversion (greatest digit elimination) from
$\{0,\ldots, b+1\}$ to $\{0, \ldots, b\}$.

\vskip0.2cm \hrule \vskip0.2cm

\noindent{\sl Input}: a finite sequence of digits $(z_j)$ of $\{0,\ldots, b+1\}$, with $z = \sum z_j\beta^j$.\\
{\sl Output}: a finite sequence of digits $\{0, \ldots, b\}$, with
$z = \sum z_j\beta^j$.

\vskip0.2cm

\noindent\texttt{for each $j$ in parallel do}\\
1. \hspace*{0.5cm} \texttt{case}
    $\left\{\begin{array}{l}
        z_j = b+1\  \\
        z_j = b\  \hbox{ \texttt{and}} \ z_{j-1} =  0 \ \\
    \end{array} \right\} $ \quad \texttt{then} $q_j:=1$\\[2mm]
\hspace*{1cm} \texttt{if} \qquad $z_j=0$ \hbox{ \texttt{and}} \ $z_{j-1} \ge b $ \qquad \ \texttt{then} $q_j:=-1$\\
\hspace*{1cm} \texttt{else} \qquad \qquad \qquad \qquad \qquad \qquad \qquad \quad \ $q_j:=0$\\
2. \hspace*{0.5cm} $z_j:=z_j-bq_j-q_{j-1}$

\vskip0.2cm \hrule \vskip0.2cm

\begin{proof}
Let $w_j=z_j-bq_j$, and $z_j^{new}= w_j-q_{j-1}$ after Step 2 of
the algorithm.
\begin{itemize}
    \item If $z_j=b+1$, then $w_j=1$. Thus $0 \le z_j^{new}\le 2 \le b$.
    \item For $z_j=b$ and $z_{j-1} = 0$, we get $w_j=0$. Since $q_{j-1}\le 0$, the resulting  $z_j^{new} \in \{0, 1\}$.
    \item For $z_j=b$ and $z_{j-1} \neq 0$, we obtain $w_j=b$. Since $q_{j-1}\neq -1$, the resulting $z_j^{new} \in \{b-1, b\}$.
    \item When $z_j=0$ and $z_{j-1} \ge b$, then $w_j=b$, and $b-1\le z_j^{new} \le b$, because $q_{j-1}\ge 0$.
    \item When $z_j=$ and $z_{j-1} \le b-1$, then $w_j=0$. Since $q_{j-1}\neq 1$, we obtain $0 \le z_j^{new}\le 1$.
    \item If $1\le z_j \le b-1$, then $0 \le z_j^{new}\le b$, as $q_j\in\{-1, 0, 1\}$.
\end{itemize}
Note that we obtain $q_j \neq 0$ only if $z_j$ itself or its
neighbor $z_{j-1}$ are different from zero; it means that the
algorithm is correct in the sense that it does not create a string
of non-zeros from a string of zeros. The input value $z$ equals
the output value $z$ thanks to the fact that the base $\beta$
satisfies $\beta^{j+1} + b \beta^j = 0$ for any $j\in\mathbb{Z}$.
This parallel conversion is $3$-local, with memory $2$ and
anticipation $0$, i.e., $(0,2)$-local since $z_j^{new}$
depends on $(z_j, z_{j-1},z_{j-2})$.
\end{proof}

Let us prove Theorem \ref{neg-par0}.

\begin{proof}  Proposition~\ref{conversion}  and the previous
Algorithm~\textsl{GDE}($-b$) imply that parallel addition is
possible in the alphabet  $\mathcal{A}_0 = \{ 0,1,\ldots, b\}$.
Moreover, Algorithm~\textsl{GDE}($-b$)  applied to the infinite
sequence $u=\lexp{\omega}h\bullet h^\omega$ gives the infinite
sequence $\varphi(u) = \lexp{\omega}h\bullet h^\omega$ for any
$h\in\{ 0,1,\ldots, b\}$. Therefore, $d$ and $b-d$ are fixed by
$\varphi$ for any $d \in\{ 0,1,2,\ldots, b\}$. Corollary
\ref{PossibleShifts} gives that parallel addition is possible on
any alphabet $\mathcal{A}_{-d} = \{-d, \ldots, , b-d\}$ for $d
\in\{ 0,1,2,\ldots, b\}$. The minimality of the alphabet
$\mathcal{A}_{-d}$ follows from Theorem~\ref{zdola}.
\end{proof}


\subsection{Base $\sqrt[k]{b}$, $b$ integer, $|b|\ge 2$}

Here we will use that $\beta$ is a zero of the polynomial $X^k-b$,
but this not in general the minimal polynomial.

\begin{proposition}\label{root-pos}
Let  $\beta=\sqrt[k]{b}$, $b$ in $\mathbb{Z}$, $|b|\ge 2$ and $k
\ge 1$ integer. Any alphabet $\mathcal{A}$ of contiguous integers
containing $0$ with cardinality  $\# \mathcal{A} = b+1$ allows
parallel addition.
\end{proposition}
The proof follows from the fact that $\gamma=\beta^k=b$ and the
results of Sections~\ref{positive} and ~\ref{negative}
applied to base $\gamma$.\\

For the sake of completeness we give below the algorithms for the
greatest digit elimination in base $\beta= \sqrt[k]{b}$, $b \ge 2$
and in base $\beta= \sqrt[k]{-b}$, $b \ge 2$.
\bigskip

\vskip0.3cm \hrule \vskip0.2cm

\noindent {\bf Algorithm~\textsl{GDE}($\sqrt[k]{b}$)}: Base
$\beta= \sqrt[k]{b}$, $b \ge 2$, parallel conversion (greatest
digit elimination) from $\{0,\ldots,b+1\}$ to $\{0, \ldots, b\}$.

\vskip0.2cm \hrule \vskip0.2cm

\noindent{\sl Input}: a finite sequence of digits $(z_j)$ of $\{0,\ldots, b+1\}$, with $z = \sum z_j\beta^j$.\\
{\sl Output}: a finite sequence of digits $\{0, \ldots, b\}$, with
$z = \sum z_j\beta^j$.

\vskip0.2cm

\noindent\texttt{for each $j$ in parallel do}\\
1. \hspace*{0.5cm} \texttt{case}
    $\left\{\begin{array}{l}
        z_j = b+1\  \\
        z_j = b\  \hbox{ \texttt{and}} \ z_{j-k} \ge b \ \\
    \end{array} \right\} $ \quad \texttt{then} $q_j:=1$\\[2mm]
\hspace*{1cm} \texttt{else} \qquad \qquad \qquad \qquad \qquad \qquad \qquad \quad \ $q_j:=0$\\
2. \hspace*{0.5cm} $z_j:=z_j-bq_j+q_{j-k}$

\vskip0.2cm \hrule \vskip0.2cm

\bigskip

\vskip0.3cm \hrule \vskip0.2cm

\noindent {\bf Algorithm~\textsl{GDE}($\sqrt[k]{-b}$)}: Base
$\beta= \sqrt[k]{-b}$, $b \ge 2$,  parallel conversion (greatest
digit elimination) from $\{0, \ldots, b+1 \}$ to $\{0, \ldots,
b\}$.

\vskip0.2cm \hrule \vskip0.2cm

\noindent{\sl Input}: a finite sequence of digits $(z_j)$ of $\{0, \ldots, b+1\}$, with $z= \sum z_j\beta^j$.\\
{\sl Output}: a finite sequence of digits $(z_j)$ of $\{0, \ldots,
b\}$, with $z= \sum z_j\beta^j$.

\vskip0.2cm

\noindent\texttt{for each $j$ in parallel do}\\
1. \hspace*{0.5cm} \texttt{case}
    $\left\{\begin{array}{l}
        z_j = b+1\  \\
        z_j = b\  \hbox{ \texttt{and}} \ z_{j-k} = 0  \ \\
    \end{array} \right\} $ \quad \texttt{then} $q_j:=1$\\[2mm]
\hspace*{1cm} \texttt{if} \qquad $z_j=0$ \hbox{ \texttt{and}} \ $z_{j-k} \ge b $ \qquad \ \texttt{then} $q_j:=-1$\\
\hspace*{1cm} \texttt{else} \qquad \qquad \qquad \qquad \qquad \qquad \qquad \quad \ $q_j:=0$\\
2. \hspace*{0.5cm} $z_j := z_j - b q_j - q_{j-k}$

\vskip0.2cm \hrule \vskip0.2cm

\bigskip

Note that in general, we cannot say that the minimal cardinality
of an alphabet for parallel addition is equal to $b+1$, since the
polynomial $X^k-b$ might be reducible. But we have the following
result. We say that $\beta = \sqrt[k]{b}$  is \emph{written in
the minimal form} if $b\neq c^{k'}$ where $k'\geq 2$ divides $k$.
Otherwise, $\beta$ could be written as  $\beta = \sqrt[k'']{c}$
with $k=k'k''$.

\begin{lemma}\label{minimal_form_positive}
Let $\beta=\sqrt[k]{b}$, with $b \in \mathbb{N}$, $b \geq 2$ and
$k$ positive integer, be written in the minimal form. Then the
polynomial $X^k-b$ is minimal for $\beta$.
\end{lemma}

\begin{proof}
 Let us suppose the opposite fact, i.e., that the polynomial    $$X^k- b = \prod_{\ell = 0}^{k-1}
\Bigl(X-e^{\tfrac{2\pi i\ell}{k}}\sqrt[k]{b}\Bigr)$$ is reducible.
One can write $ X^k- b = f(X)g(X)$, where $f(X)$ and $g(X)$ are
monic polynomials belonging to $\mathbb{Z}[X]$, the polynomial
$f(X)$ is irreducible and its degree $m$ satisfies $1\leq m<k$.
  Let $f(X) =
X^m+f_{m-1}X^{m-1}+\cdots + f_1X+f_0$.  All $m$ zeros of $f$ are
zeros of $X^k- b$ as well, i.e., of the form $ \sqrt[k]{b}$ times a
complex unit. The product of zeros of $f(X)$ is equal to
$(-1)^mf_0$, so  we have
$$|f_0|=\Bigl( \sqrt[k]{b}\Bigr)^m = b^{\frac{m}{k}}
=b^{\frac{m'}{k'}} \,$$
 where $\tfrac{m}{k} =  \tfrac{m'}{k'}$
and $m'$ and $k'$ are coprime.  Let $|f_0| = p_1^{\alpha_1}\cdots
p_r^{\alpha_r}$ be the decomposition into product of distinct
primes $p_1, \ldots, p_r$. Then
$$ b^{m'} = p_1^{k'\alpha_1}\cdots p_r^{k'\alpha_r}$$
and thus $m'$ divides $k'\alpha_j$ for all $j =1,2,\ldots, r$.
Since $k'$ and $m'$ are coprime,  $m'$ divides $\alpha_j$ and
therefore $\alpha_j = m'\alpha'_j$. We can write
$$ b =\Bigl(p_1^{\alpha'_1}\cdots p_r^{\alpha'_r}\Bigr)^{k'}=: c^{k'}.$$
As $1>\tfrac{m}{k} =  \tfrac{m'}{k'}$ the number $k'\geq 2$ and
$k'$ divides $k$ -- a contradiction with the minimal form of $\beta$.
\end{proof}

\begin{corollary}
Let $\beta=\sqrt[k]{b}$, $b$ in $\mathbb{N}$, $b\ge 2$ and $k \ge
1$ integer, written in  the minimal form. Parallel addition is
possible on any alphabet (of contiguous integers) of cardinality
$b+1$ containing $0$, and this cardinality is the smallest
possible.
\end{corollary}
\begin{proof}
Since $f(X)=X^k-b$ is the minimal polynomial of $\beta$, the lower
bound of Theorem~\ref{zdola} is equal to $|f(1)|+2=b+1$.
\end{proof}

\bigskip

We now present several cases of complex bases of the form
$\beta=\sqrt[k]{-b}$, $b$ in $\mathbb{N}$, $b\ge 2$.

The complex base $\beta = -1+\imath$ satisfies $\beta^4=-4$. Its
minimal polynomial is $f(X)=X^2+2X+2$, and the lower bound on the
cardinality of alphabet allowing parallel addition (from
Theorem~\ref{zdola}) is $|f(1)|=5$. It has been proved in
\cite{Fr} by indirect methods that parallel addition on alphabet
$\mathcal{A} = \{-2,\ldots,2\}$ is possible; and, due to
Theorem~\ref{zdola}, this alphabet is minimal.

\begin{corollary}
In base $\beta=-1+\imath$, parallel addition is possible on any
alphabet of cardinality $5$ containing $0$, and this cardinality
is the smallest possible.
\end{corollary}

\begin{remark}\label{kblock}
With a more general concept of parallelism ($k$-block $p$-local
function, see~\cite{Kornerup}), there is a result by
Herreros~\cite{Herreros} saying that addition in this base is
realizable on $\{-1,0,1\}$ by a $4$-block $p$-local function.
\end{remark}

\bigskip
The complex base $\beta=2\imath$ has $X^2+4$ for minimal
polynomial, so the lower bound given by Theorem~\ref{zdola}, equal
to $5$, is attained.

\begin{corollary}
In base $\beta=2\imath$, parallel addition is possible on any
alphabet of cardinality $5$ containing $0$, and this cardinality
is the smallest possible.
\end{corollary}

\bigskip
Similarly the complex base $\beta=\imath \sqrt{2}$ has $X^2+2$ for
minimal polynomial, so the lower bound given by
Theorem~\ref{zdola}, equal to $3$, is attained.

\begin{corollary}
In base $\beta=\imath \sqrt{2}$, parallel addition is possible on
any alphabet of minimal cardinality $3$ containing $0$.
\end{corollary}

\section{Quadratic Pisot units bases}
\subsection{Base $\beta$ root of $X^2 = a X -1$}\label{Sec_Minus}

Among the quadratic Pisot units, we firstly take as base  $\beta$
the greater zero of the polynomial $f(X)=X^2 - aX + 1$ with $a
\geq 3$. Here, the canonical alphabet of the numeration system
related to this base by means of the R\'{e}nyi expansion (greedy
algorithm) is the set $\mathcal{C} = \{ 0,  \ldots, a-1 \}$ of
cardinality $\#\mathcal{C} = a$.

The numeration system given by this base $\beta$, alphabet
$\mathcal{C}$, and the R\'{e}nyi expansions is restricted only to
representations $x = \sum_j x_j \beta^j$, where not only the
digits must be from the alphabet $\mathcal{C}$, but also the
representations must avoid any string of the form
$(a-1)(a-2)^n(a-1)$ for any $n \in \mathbb{N}$. With this
admissibility condition, the numeration system has no redundancy.
In order to enable parallel addition, we always have to introduce
some level of redundancy into the numeration system. In this case,
we prove that it is sufficient to stay in the same alphabet
$\mathcal{A} := \mathcal{C} = \{ 0,  \ldots, a-1 \}$, we only need
to cancel the restricting condition given by the R\'{e}nyi
expansion, so that all the strings on $\mathcal{C}$ are allowed.

The lower bound on the cardinality of alphabet for parallel
addition  given by Theorem~\ref{zdola} for this base is equal to
$|f(1)|+2 = a$, which is just equal to the cardinality of
$\mathcal{C}$. We show below that the canonical alphabet
$\mathcal{C} =\{0,\ldots, a-1\}$ already allows parallel addition.
At the same time, the cardinality of this alphabet $\mathcal{C}$
is equal to $\lceil \beta \rceil = a$, and thus this example
demonstrates that also the lower bound given by
Theorem~\ref{zdola2} cannot be further improved in general.


According to Proposition~\ref{conversion}, we know that, for
parallel addition on the alphabet $\mathcal{A} = \{ 0,  \ldots,
a-1 \}$, it is enough to show that parallel conversion (greatest
digit elimination) from $\{0, \ldots, a \}$ to $\mathcal{A}$ is
possible. To perform conversion from $\mathcal{A}+ \mathcal{A}$ to
$\mathcal{A}$, we then use several times greatest digit
elimination (GDE). However, the repetition of GDE may increase the
width $p$ of the sliding window in the resulting $p$-local
function. To illustrate this phenomenon,  we provide below the
complete algorithm for parallel addition, which uses GDE just once
and then in Remark \ref{discuss} we compare the value of the width
$p$ for both  approaches.

\vskip0.3cm \hrule \vskip0.2cm

\noindent {\bf Algorithm~A}: Base $\beta$ satisfying
$\beta^2=a\beta-1$,  with $a \geq 3$, parallel conversion from
$\{0,\ldots, 2a-2\}$ to $\{0, \ldots, a\}$.

\vskip0.2cm \hrule \vskip0.2cm

\noindent{\sl Input}: a finite sequence of digits $(z_j)$   of $\{0,\ldots, 2a-2\}$, with $z = \sum z_j\beta^j$.\\
{\sl Output}: a finite sequence of digits $(z_j)$ of $\{0,
\ldots,a\}$, with $z = \sum z_j\beta^j$.

\vskip0.2cm

\noindent\texttt{for each $j$ in parallel do}\\
1. \hspace*{0.5cm} \texttt{case}
    $\left\{\begin{array}{l}
        z_j \geq a\  \\
        z_j = a-1\  \hbox{ \texttt{and}} \ z_{j+1} \geq a \ \hbox{ \texttt{and}}\ z_{j-1}\geq a \ \\
    \end{array} \right\} $\ \texttt{then} $q_j:=1$\\[2mm]
   \hspace*{1cm} \texttt{else} \qquad \qquad \qquad \qquad \qquad \qquad \qquad \qquad \qquad \qquad \qquad \, $q_j:=0$\\
2. \hspace*{0.5cm} $z_j := z_j - a q_j + q_{j+1} + q_{j-1}$

\vskip0.2cm \hrule \vskip0.2cm

\begin{proof}
For correctness of Algorithm~A, we have to show that the value
$z_j^{new}= z_j- a q_j + q_{j+1} + q_{j-1}$ belongs to the
alphabet $\{0,1, \ldots, a\}$ for each $j$. Let us denote $w_j:=
z_j - a q_j$, i.e., $z_j^{new} = w_j + q_{j+1} + q_{j-1}$.
\begin{itemize}
    \item If $z_j \in \{ 0,\ldots, a-2\} \cup \{a, \ldots, 2a-2\}$,  then $w_j \in \{ 0,\ldots a-2\}$,
    and therefore $z_j^{new} = w_j + q_{j+1} + q_{j-1} \in \{ 0,\ldots a\}$.
    \item When $z_j =a-1$ and both its neighbors $z_{j\pm 1}  \geq a$,
    then $w_j=-1$ and $q_{j+1} = q_{j-1} = 1$. Thus $z_j^{new} = 1$.
    \item If $z_j = a-1$, and $z_{j-1} < a$ or $z_{j+1} < a$, then
    $w_j = a-1$ and $q_{j+1}$ or $q_{j-1}=0$. Now $z_j^{new} \in \{a-1, a\}$.
\end{itemize}
The output value $z$ equals the input value $z$ thanks to the fact
that $\beta^{j+2} - a \beta^{j+1} + \beta^j = 0$ for any $j \in
\mathbb{Z}$. Besides, it is to be noted that $z_j = 0$ implies
$q_j = 0$, and therefore, the algorithm cannot assign a string of
non-zeros to a string of zeros.
\end{proof}

We then realize the greatest digit elimination in parallel. Let us
denote by $\beta^-$ the root larger than $1$ of the equation
$X^2=aX-1$, $a \ge 3$.

\vskip0.2cm \hrule \vskip0.2cm

\noindent {\bf Algorithm~\textsl{GDE}($\beta^-$)}: Base
$\beta=\beta^-$ satisfying $\beta^2=a\beta-1$, with $a \geq 3$,
parallel conversion (greatest digit elimination) from $\{
0,\ldots, a\}$ to $\{0, \ldots, a-1\} = \mathcal{A}$.

\vskip0.2cm \hrule \vskip0.2cm

\noindent{\sl Input}: a finite sequence of digits $(z_j)$ of  $\{0,\ldots, a \}$, with $ z= \sum z_j\beta^j$.\\
{\sl Output}: a finite sequence of digits $(z_j)$ of $\{0,\ldots,
a-1\}$, with $z = \sum z_j\beta^j$.

\vskip0.2cm

\noindent\texttt{for each $j$ in parallel do}\\
1. \hspace*{0.5cm} \texttt{case}
    $\left\{\begin{array}{l}
        z_j = a\  \\
        z_j =a-1\  \hbox{ \texttt{and}} \ \bigl( z_{j+1} \geq  a-1 \ \hbox{ \texttt{or}}\ z_{j-1}\geq a -1\Bigr) \ \\
        z_j =a-2\  \hbox{ \texttt{and}} \ z_{j+1} = a \ \hbox{ \texttt{and}} \ z_{j-1} = a\ \\
        z_j =a-2\  \hbox{ \texttt{and}} \ z_{j+1} = a \ \hbox{ \texttt{and}} \ z_{j-1} = a-1 \ \hbox{ \texttt{and}} \ z_{j-2} \geq a-1 \\
        z_j =a-2\  \hbox{ \texttt{and}} \ z_{j-1} = a \ \hbox{ \texttt{and}} \ z_{j+1}=a-1 \ \hbox{ \texttt{and}} \ z_{j+2} \geq a-1 \\
        z_j =a-2\  \hbox{ \texttt{and}} \ z_{j \pm 1} = a-1 \ \hbox{ \texttt{and}} \ z_{j \pm 2} \geq a-1 \ \\
    \end{array} \right\} $\\ \hspace*{1cm} \texttt{then} $q_j:=1$\\[1mm]
   \hspace*{1cm} \texttt{else} $q_j:=0$\\[2mm]
2. \hspace*{0.5cm} $z_j := z_j - a q_j + q_{j+1} + q_{j-1}$

\vskip0.2cm \hrule \vskip0.2cm

\begin{proof}
Let us denote   again $w_j := z_j- a q_j$, i.e., $z_j^{new} =
w_j + q_{j+1} + q_{j-1}$.
\begin{itemize}
    \item If $z_j \in \{ 0,\ldots, a-3\} \cup \{a\}$,  then $w_j \in \{0,\ldots a-3\}$,
     and therefore $z_j^{new} = w_j + q_{j+1} + q_{j-1} \in \{ 0,\ldots a-1\} = \mathcal{A}$.
    \item When $z_j = a-1$, and $z_{j - 1} \geq a-1$ or $z_{j + 1} \geq a-1$,
    then $w_j=-1$ and $q_{j+1} + q_{j-1} \in \{1,2\}$. Thus $z_j^{new}\in \{0,1\} \subset \mathcal{A}$.
    \item When $z_j = a-1$ and both its neighbors $z_{j \pm 1} < a-1$,
    then $w_j = a-1$ and $q_{j+1} = q_{j-1} = 0$. Now $z_j^{new} = a-1 \in \mathcal{A}$.
    \item If $z_j = a-2$ and $q_j =1$, then necessarily $q_{j \pm 1} = 1$.
    Since  $w_j=-2$, we get $z_j^{new}=0 \in \mathcal{A}$.
    \item If $z_j = a-2$ and $q_j =0$, then necessarily $q_{j-1}$ or $ q_{j+1}$
    equal $0$, and therefore $q_{j+1} + q_{j-1} \in \{0,1\}$. As $w_j=a-2$,
     the resulting $z_j^{new}\in \{a-2, a-1\} \subset \mathcal{A}$.
\end{itemize}
Again, the equation $\beta^{j+2} - a \beta^{j+1} + \beta^j = 0$
for  any $j \in \mathbb{Z}$ ensures that the output value $z$
equals the input value $z$. For $z_j = 0$ we always have $q_j =
0$, so the algorithm cannot assign a string of non-zeros to a
string of zeros.
\end{proof}

Now we can proceed by summarizing the algorithm for parallel
addition:

\vskip0.3cm \hrule \vskip0.2cm

\noindent {\bf Algorithm I}: Base $\beta$ satisfying
$\beta^2=a\beta-1$,  with $a \geq 3$, parallel addition on
alphabet $\mathcal{A} = \{0, \ldots, a-1\}$.

\vskip0.2cm \hrule \vskip0.2cm

\noindent{\sl Input}: two finite sequences of digits $(x_j)$ and
$(y_j)$
 of $\{0, \ldots, a-1 \}$, with $x = \sum x_j\beta^j$ and $y = \sum y_j\beta^j$.\\
{\sl Output}: a finite sequence of digits $(z_j)$ of $\{0,\ldots,
a-1\}$  such that $z = x+y = \sum z_j\beta^j$.

\vskip0.2cm

\noindent\texttt{for each $j$ in parallel do}\\
0. \hspace*{0.5cm} $v_j:=x_{j} +y_{j}$\\
1. \hspace*{0.5cm} use Algorithm~A with input $(v_j)$ and denote its output $(w_j)$\\
2. \hspace*{0.5cm} use Algorithm~\textsl{GDE}($\beta^-$) with
input $(w_j)$ and denote its output $(z_j)$

\vskip0.2cm \hrule \vskip0.2cm

\smallskip

\begin{theorem}\label{b-pos}
Let $\beta >1$ be a root of $X^2=aX-1$, with $a \geq 3$, $a \in
\N$, and let $\mathcal{A}$ be the canonical alphabet $\mathcal{A}
= \{0,\ldots,a-1\}$ associated with this base $\beta$. Addition in
${\rm Fin}_{\mathcal{A}}(\beta)$ is a $p$-local function with
$p=11$. The alphabet $\mathcal{A}$ is the smallest one for
parallel addition.
\end{theorem}

\begin{proof}
In Algorithm~A, the output digit $z_j^{new}$ depends on input
digits  $(z_{j+2}, \ldots, z_{j-2})$, so it is a $(2,2)$-local
function. In Algorithm~\textsl{GDE}($\beta^-$) the output digit
$z_j^{new}$ depends on input digits $(z_{j+3}, \ldots, z_{j-3})$,
and thus it is a $(3,3)$-local function. Algorithm~I is a
composition of Algorithms~A and \textsl{GDE}($\beta^-$), so the
resulting function is a composition of the two local functions,
$(2,2)$-local and $(3,3)$-local. Overall, the addition in base
$\beta$, fulfilling $\beta^2 = a\beta - 1$ for $a \geq 3$, as
performed by Algorithm~I, is a $(5,5)$-local function, i.e.,
$11$-local.
\end{proof}

\begin{remark}\label{discuss}
According to Proposition~\ref{conversion}, we need only
Algorithm~\textsl{GDE}($\beta^-$) for performing parallel addition
on $\mathcal{A}=\{0,\ldots, a-1\}$  in the base $\beta^2 = a\beta
-1$, with $a \geq 3$. In order to obtain the sum $x+y$, we would
apply Algorithm~\textsl{GDE}($\beta^-$)
 repeatedly $(a-1)$-times. The function performing parallel addition
 in this way would then be $(3a-3,3a-3)$-local. On the other hand,
 Algorithm~I, exploiting firstly Algorithm~A and then only once the
 Algorithm~\textsl{GDE}($\beta^-$), reduces the size of the sliding
  window, i.e., the parameters of the local function are only $(5,5)$.
\end{remark}


Now we are going to show that parallel addition for base  $\beta^2
= a\beta - 1$, with $a \geq 3$, $a \in \N$, is feasible also on
any alphabet of contiguous integers of cardinality $a$ containing
$\{-1,0,1\}$, of the form $\mathcal{A}_{-d} = \{-d, \ldots,  0,
\ldots, a-1-d\}$, for $1 \le d \le a-2$.

Let us observe that Algorithm~\textsl{GDE}($\beta^-$) applied to
the bi-infinite sequence $u=\lexp{\omega}h\bullet h^\omega$ gives
the bi-infinite sequence $\varphi(u) = \lexp{\omega}h\bullet
h^\omega = u$ for any $h\in\{ 0,\ldots, a-2\}$, and thus for any
$d \in \{1,\ldots,d-2\}$, both letters $d$ and $a-1-d$ are fixed
by $\varphi$. Corollary \ref{PossibleShifts} therefore implies
that the alphabet $\mathcal{A}_{-d} =\{-d, \ldots, a-1-d\}$ allows
parallelization of addition for any such $d \in \{1, \ldots,
d-2\}$. This fact, together with Theorem~\ref{zdola},
Proposition~\ref{opposite}, and Theorem~\ref{b-pos} enable us to
conclude this section with the following theorem.

\begin{theorem}
Let $\beta$ satisfy $\beta^2 = a\beta -1$, with $a \in \mathbb{N},
a \geq 3$. Parallel addition is possible on any alphabet of
contiguous integers containing $0$ of cardinality $a$, and this
cardinality is minimal.
\end{theorem}


\subsection{Base $\beta$ root of $X^2 =a X +1 $}\label{Sec_Plus}

Let us now study the numeration systems with base a quadratic
Pisot unit with  minimal polynomial  $f(X)=X^2 - aX -1$, with $a
\geq 1$. The canonical alphabet of the numeration system related
with this base by means of the R\'{e}nyi expansion (greedy
algorithm) is $\mathcal{C} = \{ 0, \ldots, a \}$, of cardinality
$\#\mathcal{C} = a+1$. The numeration system given by this type of
quadratic base $\beta$, alphabet $\mathcal{C}$, and the R\'{e}nyi
expansions is restricted only to such representations $x = \sum_j
x_j \beta^j$, where the digits are from the alphabet
$\mathcal{C}$, but the representations must avoid any string of
the form $a1$. This admissibility condition makes the numeration
system non-redundant.

It is known, for bases $\beta$ satisfying $\beta^2 = a\beta + 1$
with $a \geq 1$,  that the set of real numbers with finite greedy
expansion $\langle x\rangle_\beta$ is closed under addition and
subtraction~\cite{BFGK}. Therefore,
$$\{ x\geq 0\ | \ \langle x\rangle_\beta \ \hbox {is finite}\ \} =
Fin_{\mathcal{A}}(\beta) \quad \hbox{for any} \ \mathcal{A}
\subset \mathbb{N}, \, \mathcal{A} \supset \mathcal{C}$$ and
$$\{ x \in \mathbb{R}\ | \ \langle |x|\rangle_\beta \ \hbox {is finite}\ \} =
Fin_{\mathcal{A}}(\beta) \quad \hbox{for any} \ \mathcal{A}
\subset \mathbb{Z}, \, \mathcal{A} \supset \mathcal{C}
\cup\{-1\}.$$

In order to obtain an algorithm for parallel addition, we must
have some  redundancy in the numeration system. As shown in
Section~\ref{Sec_Minus}, for the cases of base $\beta$ satisfying
$\beta^2 = a \beta - 1$, it was sufficient to drop the one
admissibility condition (given by the R\'{e}nyi expansion), and
parallelization was already possible (without adding any more
elements into the alphabet $\mathcal{C}$). The situation is not
that simple for the bases satisfying $\beta^2 = a \beta +1$.
Nevertheless, addition in these two families of quadratic units is
connected.

\begin{proposition}\label{Prop_From_Plus_To_Minus}
Let $\beta>1$ be a zero of the polynomial $X^2-aX-1$ with $a \geq
1$, and let $\gamma>1$ be a zero of the polynomial
$X^2-(a^2+2)X+1$. If there exists an alphabet $\mathcal{A}$ and a
$p$-local function performing  in ${\rm Fin}_\mathcal{A}(\gamma)$
addition in parallel, then there exists a $(2p-1)$-local function
performing in  ${\rm Fin}_\mathcal{A}(\beta)$  addition in
parallel.
\end{proposition}

\begin{proof}
It is enough to realize that $\gamma=\beta^2$, and to apply
Theorem~1 from \cite{Fr}.
\end{proof}

\begin{remark}\label{upper_bound}
According to the previous Section~\ref{Sec_Minus}, we know that
addition in base $\gamma$, the zero of the polynomial $X^2
-(a^2+2)X +1$, can be performed in parallel on alphabet $\{-d,
\ldots, a^2+1-d\}$ for any $d \in \{ 0, \ldots, a^2\}$. Therefore,
we immediately obtain an upper bound  $a^2+2$ on the cardinality
of the alphabet allowing parallel addition in base $\beta$, the
zero of the polynomial $X^2 -aX -1$.
\end{remark}

In general, the upper bound given in Remark~\ref{upper_bound} is
too rough. But for $a=1$, i.e., for the base the golden
ratio, it gives the precise value of cardinality of the minimal
alphabet for parallel addition in this base, namely the
cardinality $\#\mathcal{A} = 3$.

\begin{corollary}\label{Golden_ratio}
Let $\beta = \frac{1+\sqrt{5}}{2}$ be the golden ratio, zero of
$X^2-X-1$. Addition in this base $\beta$ can be performed in
parallel on alphabet $\mathcal{A} = \{ 0,1,2\}$, and also on
alphabet $\mathcal{A} = \{-1,0,1\}$.  Both these alphabets are
minimal.
\end{corollary}

Let us mention that this result for the alphabet $\{-1, 0, 1\}$ is
already stated in \cite{FrPeSv}. Non-sufficiency of the alphabet
$\{0,1\}$ for
parallel addition is stated in \cite{Fr}.\\

In the sequel, we are going to consider only parameters $a \geq
2$. The lower bound on the cardinality of the alphabet of
contiguous integers allowing parallel addition, given by
Theorem~\ref{zdola} for bases $\beta$ being zeros of polynomials
$X^2-aX-1$, is equal to $|f(1)|+2=a+2$. We show that,
 in these cases, parallel addition is doable on any alphabet of contiguous
 integers containing $0$ of cardinality $a+2$.

For short, the positive zero of $X^2-aX-1$ is denoted by
$\beta^+$.
\vskip0.3cm \hrule \vskip0.2cm

\noindent {\bf Algorithm~\textsl{GDE}($\beta^+$)}: Base
$\beta=\beta^+$ satisfying $\beta^2=a\beta+1$, $a \ge 2$, $a\in
\mathbb{N}$, parallel conversion
 (greatest digit elimination) from $\{0,\ldots, a+2\}$ to $\mathcal{A} = \{0, \ldots, a+1\}$.

\vskip0.2cm \hrule \vskip0.2cm

\noindent{\sl Input}: a finite sequence of digits $(z_j)$ of
$\{0,\ldots, a+2\}$, with $z = \sum z_j\beta^j$.\\
{\sl Output}: a finite sequence of digits $(z_j)$ of $\{0, \ldots,
a +1\}$, with $z = \sum z_j\beta^j$.

\vskip0.2cm

\noindent\texttt{for each $j$ in parallel do}\\
1. \hspace*{0.5cm} \texttt{case}
    $\left\{\begin{array}{l}
        z_j = a+2\ \\
        z_j =a+1\ \hbox{ \texttt{and}} \ (z_{j+1} = 0 \ \hbox{ \texttt{or}} \ z_{j-1}\geq a+1 ) \ \\
        z_j =a\ \hbox{ \texttt{and}} \ z_{j+1} = 0 \ \hbox{ \texttt{and}} \ z_{j-1}\geq a+1 \ \\
    \end{array} \right\} $\ \texttt{then} \, $q_j:=1$\\
   \hspace*{1cm} \texttt{if} \qquad \ $z_j=0$ \texttt{and} $z_{j+1} \ge a+1$ \texttt{and} $z_{j-1}\leq  a$ \qquad \qquad \ \texttt{then} \, $q_j:=-1$\\[2mm]
   \hspace*{1cm} \texttt{else} \qquad \qquad \qquad \qquad \qquad \qquad \qquad \qquad \qquad \qquad \qquad \qquad \ \,\, $q_j:=0$\\
2. \hspace*{0.5cm} $z_j := z_j - a q_j - q_{j+1} + q_{j-1}$

\vskip0.2cm \hrule \vskip0.2cm

\begin{proof}
Let us denote again $w_j = z_j- a q_j$, and $z_j^{new} = w_j -
q_{j+1} + q_{j-1}$.
\begin{itemize}
    \item If $z_j=a+2$, then $w_j=2$. Since $q_{j+1} \geq 0$, we
     have $-q_{j+1}+q_{j-1} \in \{-2, \ldots, 1\}$, and consequently,
     $z_j^{new} \in \{0, \ldots, 3\} \subset \{0, \ldots, a+1\} = \mathcal{A}$, using the fact that $a \ge 2$.
    \item For $z_j=a+1$ and $z_{j+1}=0$, we get $w_j=1$. As $q_{j+1}=0$,
    then $z_j^{new} \in \{0, 1, 2\} \subset \mathcal{A}$.
    \item For $z_j=a+1$ and $z_{j-1} \geq a+1$, we obtain again $w_j=1$.
    Since $q_{j-1} \ge 0$, then $-q_{j+1}+q_{j-1} \in \{-1, \ldots, 2\}$,
     and consequently, $z_j^{new} \in \{0, \ldots, 3\} \subset \{0, \ldots, a+1\} = \mathcal{A}$, as $a \geq 2$.
    \item If $z_j=a+1$ and $z_{j+1} \geq 1$ and $z_{j-1} \leq a$,
    then $w_j=a+1$, $q_{j-1} \le 0$ and $q_{j+1} \geq 0$. Therefore, $z_j^{new} \in \{a-1, a, a+1\} \subset \mathcal{A}$.
    \item In the case of $z_j=a$ and $z_{j+1}=0$ and $z_{j-1} \geq a+1$,
    we obtain $w_j=0$. Since $q_{j+1} \le 0$ and $q_{j-1} \ge 0$, the
    resulting $z_j^{new} \in \{0, 1, 2\} \subset \mathcal{A}$.
    \item When $z_j=a$ and $z_{j+1} \geq 1$, then $w_j=a$. Since
    $q_{j+1}\ge 0$ and $q_{j-1} \ge 0$, we obtain $z_j^{new} \in \{a-1, a, a+1\} \subset \mathcal{A}$.
    \item When $z_j=a$ and $z_{j-1} \leq a$, then again $w_j=a$.
    This time, $q_{j-1}= 0$, so consequently, $z_j^{new} \in \{a-1, a, a+1\} \subset \mathcal{A}$.
    \item If $z_j=0$ and $z_{j+1} \geq a+1 $ and $z_{j-1} \leq a$,
    then $w_j=a$. Since $q_{j+1}\ge 0$ and $q_{j-1} \geq  0$, we obtain $z_j^{new} \in \{a-1, a, a+1\} \subset \mathcal{A}$.
    \item For $z_j=0$ and $z_{j+1} \leq a $, we get $w_j=0$, $q_{j+1} \leq 0$,
    and $q_{j-1} \ge 0$. Therefore, the resulting $z_j^{new} \in \{0, 1, 2\} \subset \mathcal{A}$.
    \item If $z_j=0$ and $z_{j-1} \ge a+1$, then $w_j=0$ and $q_{j-1}=1$.
    Consequently, $z_j^{new} \in \{0, 1, 2\} \subset \mathcal{A}$.
    \item For the cases when $z_j \in \{1, \ldots, a-1\}$, we have $q_j=0$
    and $q_{j-1} \geq 0$, so consequently, $z_j^{new} \in \{0, \ldots, a+1\} = \mathcal{A}$.
\end{itemize}
Note that, for $z_j=0$, we can only obtain $q_j \neq 0$ when its
neighbor $z_{j+1}$ is greater than zero. Therefore, it is ensured
that the algorithm cannot assign a string of non-zeros to a string
of zeros. The output value $z$ is equal to the input value $z$
thanks to the fact that the base $\beta$ satisfies the equation
$\beta^{j+2} = a \beta^{j+1} + \beta^j$ for any $j \in
\mathbb{Z}$. The output digit $z_j^{new}$ depends on input digits
$(z_{j+2}, \ldots, z_{j-2})$, so this conversion from $\{0,
\ldots, a+2\}$ to $\mathcal{A} = \{0, \ldots, a+1\}$ is a
$(2,2)$-local function.
\end{proof}

Using Proposition \ref{conversion}  we can conclude that addition
in ${\rm Fin}_{\mathcal{A}}(\beta)$ can be performed
 in parallel on the alphabet $\{0,1,\ldots,
 a+1\}$.

\bigskip

Let us now consider alphabet containing positive and negative
digits. For any
 $d \in \{1, \ldots, a\}$, denote
 $$\mathcal{A}_{-d} = \{-d, \ldots, 0, \ldots, a-d +1\}\,.$$

The previous Algorithm~\textsl{GDE}($\beta^+$) applied to the
infinite sequence $u=\lexp{\omega}h\bullet h^\omega$ gives the
infinite sequence $\varphi(u) = \lexp{\omega}h\bullet h^\omega =
u$ for any $h\in\{ 0,\ldots, a\}$. Thus, for any $d \in \{
1,2,\ldots, a\}$, both letters $d$  and $a+1-d$ are fixed by
$\varphi$. According to Corollary~\ref{PossibleShifts}, the
alphabet $\mathcal{A}_{-d}$ allows parallelism of addition.
Summarizing this reasoning, together with Algorithm
GDE($\beta^+$), Corollary~\ref{Golden_ratio}, and
Proposition~\ref{opposite}, we obtain the following result.
\begin{theorem}
Let $\beta$ satisfy $\beta^2 = a\beta +1$, with $a \geq 1, a\in
\mathbb{N}$. Parallel addition is possible on any alphabet of
contiguous integers containing $0$, such that its cardinality is
$a+2$. The cardinality $a+2$ is minimal.
\end{theorem}


\section{Rational Bases}\label{rat}

Let us now consider the base $\beta=\pm a/b$, with $a$, $b$ being
co-prime positive integers fulfilling $a > b \ge 1$. When $b=1$,
we obtain the   case of positive integer base $\beta = a \in
\mathbb{N}$, $a \geq 2$, or the case of negative integer base
$\beta = -a \in \mathbb{N}$, $a \geq 2$ with the minimal
cardinality of alphabet for parallel addition being equal to
$a+1$, see Sections~\ref{positive} and \ref{negative}.
 For $b \ge 2$, the base $\beta$ is an algebraic number which is
not an algebraic integer, so Theorem~\ref{zdola} cannot be applied
here to establish a lower bound on the cardinality of alphabet for
parallel addition. Theorem~\ref{zdola2} can be used for $\beta=
a/b$, however it is not very useful here either; the lower bound
given there is
   equal to $\lceil  a/b \rceil$, which is  too low for parallel addition, as is shown below.

In general, an alphabet $\mathcal{A}$ allows parallel addition
only if the numbers with finite representation are closed under
addition, in particular, any non-negative integer must  have a
finite representation. This requirement already forces the
alphabet to be big enough. By a modification of the Euclidean
division algorithm, any non-negative integer can be given a unique
finite expansion in base $\beta=a/b$, and any integer can be given
a unique finite expansion in base $\beta=-a/b$, on the alphabet
$\mathcal{C}=\{0,\ldots,a-1\}$, see~\cite{FrougnyKlouda} and
\cite{AkiFroSaka}.

As we shall see, even this alphabet is to small. For both positive
base $\beta= a/b$ and  negative base $\beta= -a/b$, the
cardinality of alphabet  actually needed for parallel addition  is
at least $a+b$. In the alphabet $\mathcal{A}=\{0,1,\ldots,
a+b-1\}$, we can perform parallel addition in the base $\beta=
a/b$ and $\beta= -a/b$ as well.

But surprisingly,  these two types of bases differ substantially
if we consider alphabets containing $\{-1,0,1\}$.


\subsection{Positive Rational Base}\label{rat-pos}


\begin{proposition}\label{rat-par}
Parallel addition in base $\beta=a/b$, with $a$ and $b$ co-prime
positive integers such that $a > b \ge 1$, is possible on
$\mathcal{A} = \{0,\ldots,a+b-1\}$.
\end{proposition}
\begin{proof}
We give a parallel algorithm Algorithm~\textsl{GDE}($a/b$):
$\{0,\ldots, a+b \} \rightarrow \{0,\ldots, a+b -1\}$ for greatest
digit elimination.
\vskip0.3cm \hrule \vskip0.2cm

\noindent {\bf Algorithm~\textsl{GDE}($a/b$)}: Base $\beta= a/b$,
with $a > b \geq 1$, $a$ and $b$ co-prime positive integers,
parallel conversion (greatest digit elimination) from $\{0,\ldots,
a+b \}$ to $\{0,\ldots, a+b -1\}$.

\vskip0.2cm \hrule \vskip0.2cm

\noindent{\sl Input}: a finite sequence of digits $(z_j)$ of $\{0,\ldots, a+b \}$, with $z= \sum z_j\beta^j$.\\
{\sl Output}: a finite sequence of digits $(z_j)$ of $\{0,\ldots,
a+b -1\}$, with $z= \sum z_j\beta^j$.

\vskip0.2cm

\noindent\texttt{for each $j$ in parallel do}\\
1. \hspace*{0.5cm} \texttt{if} $ a \le z_j \le a+b$\ \qquad \texttt{then} $q_j:=1$\\[2mm]
   \hspace*{1cm} \texttt{else}   \qquad \qquad \qquad  \qquad \qquad $q_j:=0$\\
2. \hspace*{0.5cm} $z_j := z_j - a q_j + b q_{j-1}$

\vskip0.2cm \hrule \vskip0.2cm

Denoting $w_j:=z_j-aq_j$, we clearly obtain $0 \le w_j \le a-1$.
Thus, after Step 2 of the algorithm, $z_j^{new} = w_j + b q_{j-1}$
belongs to  $\{0,\ldots, a+b-1\}$. This algorithm assigns $q_j
\neq 0$ only in the cases of $z_j \neq 0$, so it cannot produce a
string of non-zeros from a string of zeros. The output value $z$
equals the input value $z$ thanks to the fact that $b \beta^{j+1}
- a \beta^j = 0$ for
 any $j \in \mathbb{Z}$. So the algorithm is correct. \\
Thus the result follows from Algorithm~\textsl{GDE}($a/b$) and
Proposition~\ref{conversion}.
\end{proof}

\begin{proposition}
In base $\beta=a/b$, with $a$ and $b$ co-prime positive integers
such that $a > b \ge 1$, parallel addition is possible on any
alphabet of cardinality $a+b$ $\mathcal{A}_{-d}=\{-d,
\ldots,0,\ldots, a+b-d-1 \}$ with $ b \leq d \leq a-1$.
\end{proposition}
\begin{proof}
Algorithm~\textsl{GDE}($a/b$)  applied to the bi-infinite sequence
$u=\lexp{\omega}h\bullet h^\omega$ gives the bi-infinite sequence
$\varphi(u) = \lexp{\omega}h\bullet h^\omega$ for any $h\in\{
0,\ldots, a-1\}$. Thus for any $d \in \{ b,b+1,\ldots, a-1\}$,
both letters $d$  and $a+b-1-d$ are fixed by $\varphi$. According
to Corollary~\ref{PossibleShifts}, the alphabet $\mathcal{A}_{-d}$
allows parallelism of addition.
\end{proof}

So the question is now: what happens for alphabets
$\mathcal{A}_{-d}$ when $d \ge a$ or $d \le b-1$?
First recall a well known fact.\\
\textbf{Fact 1}. Let $\lexp{\omega}0c_k \cdots c_0 \bullet c_{-1}
\cdots c_{-\ell}0^\omega $ and $\lexp{\omega}0d_k \cdots d_0
\bullet d_{-1} \cdots d_{-\ell}0^\omega  $, $k, \ell \ge 0$, be
two representations in base $\beta=a/b$ of the same number in
$\Z[\beta]$. Then the polynomial $(c_k-d_k)X^k + \cdots +
(c_0-d_0) + \cdots + (c_{-\ell}-d_{-\ell})X^{-\ell}$ is a multiple
of $bX-a$. Thus there exist $s_{k-1}, s_{k-2}, \ldots, s_0,
s_{-1},\ldots, s_{-\ell} \in \Z$ such that
\begin{equation}\label{fact1}
c_k-d_k=bs_{k-1}\,,
\end{equation}
\begin{equation}\label{fact2}
c_j-d_j=-as_j+bs_{j-1} \quad \hbox{for any} \ k-1 \ge j \ge
-\ell+1\,,
\end{equation}
\begin{equation}\label{fact3}
c_{-\ell}-d_{-\ell}=-as_{-\ell}\,.
\end{equation}

\begin{lemma}\label{non-local2}
Let $\mathcal{D}=\{m,\ldots,0,\ldots,M\}$ with $m \le -1$ and
$M\ge 1$, be an alphabet. If $M<b$ then the greatest digit
elimination from $\mathcal{D} \cup \{M+1\}$ to $\mathcal{D}$ is
not a local function; if $m >-b$, then the smallest digit
elimination from $ \{m-1\}\cup\mathcal{D}$ to $\mathcal{D}$ is not
a local function either.
\end{lemma}
\begin{proof}
Suppose that $M<b$ and that the greatest digit elimination from
$\mathcal{D} \cup \{M+1\}$ to $\mathcal{D}$ is a $p$-local
function $\varphi$. Consider the digit $M+1$, and suppose that
$M+1$ has a representation on $\mathcal{D}$, of the form
$\lexp{\omega}0d_{k} \cdots d_0\bullet d_{-1}\cdots
d_{-\ell}0^\omega$, with $0 <k $, $0< \ell$ (the values $d_{k}=0$
and $d_{-\ell} =0$ are not excluded). By Fact 1, there exist
integers $s_j$ such that
\begin{itemize}
  \item $d_0= M+1+as_0-bs_{-1}$
  \item for $1 \le j \le k-1$, $d_j= as_j-bs_{j-1}$
    \item $d_k=-bs_{k-1}$
  \item for $1 \le j \le \ell-1$, $d_{-j}= as_{-j}-bs_{-j-1}$
  \item $d_{-\ell}= as_{-\ell}$.
\end{itemize}
Since $d_k=-bs_{k-1} \in \mathcal{D}$ and $M<b$, $s_{k-1} \ge 0$.
Then $d_{k-1}= as_{k-1}-bs_{k-2} \ge -bs_{k-2}\in \mathcal{D}$
implies that $s_{k-2} \ge 0$. Similarly, $s_{k-3} \ge 0$, \ldots,
$s_0 \ge 0$. Since $M\geq d_0= M+1+as_0-bs_{-1} \ge M+1 -bs_{-1}
$, we must have $1 -bs_{-1}\le 0$, hence $ s_{-1} \geq 1$. On the
other hand, $b> d_{-\ell}= as_{-\ell}\in \mathcal{D}$ implies that
$ s_{-\ell} \le 0$.  Then $b> d_{-\ell+1}= as_{-\ell+1}-bs_{-\ell}
\ge as_{-\ell+1}$ implies
$s_{-\ell+1} \le 0$. Similarly, $s_{-\ell+2} \le 0$, \ldots, $s_{-1} \le 0$, a contradiction.\\
The case $m <-b$ is analogous.
\end{proof}

\begin{corollary}
In base $\beta=a/b$, with $a$ and $b$ co-prime positive integers
such that $a > b \ge 1$, parallel addition is not possible on
alphabets of positive and negative digits not containing $\{-b,
\ldots,0, \ldots,b\}$.
\end{corollary}
Note that in~\cite{FrPeSv} we have given an alphabet of the form
$\{-d, \ldots,0, \ldots,d\}$ on which parallel addition in base
$a/b$ is possible, with $d=\lceil \frac{a-1}{2}\rceil  +b$.

\bigskip

\begin{proposition}\label{non-loc}
Let $a$ and $b$ be co-prime positive integers such that $a > b \ge
1$.  The minimal alphabet of contiguous non-negative integers
containing $0$ allowing parallel addition in base $\beta=a/b$ is
$\mathcal{A} = \{0,\ldots,a+b-1\}$.
\end{proposition}

\begin{proof}  Let
us suppose that this statement is not valid, it means that there
exists a $p$-local function $\varphi: \mathcal{A}^\mathbb{Z} \to
\mathcal{B}^\mathbb{Z}$ which performs conversion in base $\beta$
from $\mathcal{A}$ to $\mathcal{B}$, where $\mathcal{B} =
\{0,\ldots,a+b-2\}$.

Let us fix  $ n \in \mathbb{N}$ and $q\in \mathbb{N}$ such that
$n>p$ and $\left(\frac{a}{b}\right)^q\geq \frac{a+b}{a-b}$.
 Then the image of $\lexp{\omega}0(a-1)^n\bullet
0^\omega$ by $\varphi$ can be written in the form
\begin{equation}\label{image1}\varphi( \lexp{\omega}0(a-1)^n\bullet 0^\omega) =
\lexp{\omega}0w_{h}w_{h-1}w_0\bullet w_{-1}w_{-2}\ldots w_{-\ell}
0^\omega\,,\end{equation}
 where  $w_h> 0$,  $w_{-\ell} \geq 0$ and  $ \ell
\geq 1$.

Consider now the string $\lexp{\omega}0(a-1)^n\bullet
(a+b-1)(a+b-2)^q0^\omega$.  Since $\varphi$ is a $p$-local
function and $n>p$, the image of this string coincides on the
positions  $j\geq p$ with the image of the string
$\lexp{\omega}0(a-1)^n\bullet 0^\omega$. Therefore we can write
\begin{equation}\label{image2}\varphi( \lexp{\omega}0(a-1)^n\bullet
(a+b-1)(a+b-2)^q0^\omega) = \lexp{\omega}0v_{h}v_{h-1}v_0\bullet
v_{-1}v_{-2}\ldots v_{-m} 0^\omega\,,
\end{equation}
where  $w_h= v_h> 0$, $v_{-m} \geq 0$  and $m > q+1$.

We will discuss the value of the index $h$ in the above
equalities.

\begin{description}

\item[Case $h\geq n$] \quad At first we focus on the equality
\eqref{image1} and apply  Fact 1 to the string
$\lexp{\omega}0(a-1)^n\bullet 0^\omega$ in the role of
$\lexp{\omega}0c_k \cdots c_0 \bullet c_{-1} \cdots
c_{-\ell}0^\omega $    and to the string
$\lexp{\omega}0w_{h}w_{h-1}w_0\bullet w_{-1}w_{-2}\cdots w_{-\ell}
0^\omega$ in the role of $\lexp{\omega}0d_k \cdots d_0 \bullet
d_{-1} \cdots d_{-\ell}0^\omega$ with  $k=h$.  As both strings
belong to  $\mathcal{B}^\mathbb{Z}$,  we obtain $0> -w_{h} =
bs_{h-1}$ which gives $s_{h-1}\leq -1$. By the same reason,  we
have  for all $j$ such that
  $ -\ell+1\leq j \leq h-1$ the inequality $$ a-1\geq  c_j-d_j = -as_j +bs_{j-1}$$  which gives
  the following
  implication
$$ s_j \leq -1 \quad \Longrightarrow \quad s_{j-1} \leq -1\,,\quad \hbox{for}\ \ -\ell+1\leq j \leq h-1\,.$$
In particular, $s_{-\ell} \leq -1$. Together  with \eqref{fact3},
 we have
$$ 0 \geq - d_{-\ell} = -as_{-\ell} \geq a \quad \hbox{- a contradiction.}
$$

\item[Case $h\leq n-1$] \quad Now we focus on the equality
\eqref{image2} and apply  Fact 1 to the string
$\lexp{\omega}0(a-1)^n\bullet (a+b-1)(a+b-2)^q0^\omega$ in the
role  of $\lexp{\omega}0c_k \cdots c_0 \bullet c_{-1} \cdots
c_{-\ell}0^\omega $    and to the string
$\lexp{\omega}0v_{h}v_{h-1}v_0\bullet v_{-1}v_{-2}\cdots v_{-m}
0^\omega$ in the role of $\lexp{\omega}0d_k \cdots d_0 \bullet
d_{-1} \cdots d_{-\ell}0^\omega$ with  $k=n-1$ and $\ell =m >
q+1$.  For the index $j=n-1$, Equality \eqref{fact1}  implies $
-b+1 \leq a-1 - d_{n-1} = bs_{n-2}$ and thus $s_{n-2} \geq 0$. For
indices  $j$, where  $n-2\geq j\geq 0$, Equality \eqref{fact2}
gives
$$ -b+1 \leq a-1-d_{j} = -a s_j + bs_{j-1} \leq a-1\,.$$
From the above  inequality, we can  derive the implication
$$ s_j \geq 0 \quad \Longrightarrow \quad s_{j-1} \geq 0\,,\quad \hbox{for}\ \ 0\leq j \leq n-2\,.$$
In particular, $s_{-1} \geq 0$. For the index $j=-1$,  Equality
\eqref{fact2} gives $$1\leq a+b-1-d_{-1} = -as_{-1}+bs_{-2}\quad
\hbox{ and thus } \ \ s_{-2}\geq 1\,.$$ For indices $-2\geq j \geq
-q-1$, we obtain
 $$0\leq a+b-2-d_j = -as_j+bs_{j-1} \quad
\Longrightarrow \quad  s_{j-1}\geq \tfrac{a}{b}s_{j}\,.$$
 In
particular,
 \begin{equation}\label{big}s_{-q-2} \geq \left(\tfrac{a}{b}\right)^qs_{-2}\geq
\left(\tfrac{a}{b}\right)^q\,.
\end{equation}
On the other hand, for the index $-\ell < -q-1$, Equality
\eqref{fact3} sounds  $-as_{-\ell} = -d_{-\ell} \geq -a-b+2$, and
thus $s_{-\ell}\leq 1$. For indices $j$ with $-\ell +1 \leq j\leq
-q-2$, one can deduce
$$ -a-b< -d_j = -as_j+bs_{j-1}\leq 0 \qquad  \Longrightarrow  \qquad s_j <1+ \tfrac{b}{a}
+\tfrac{b}{a}s_{j-1}\,.$$ The last  inequality  enables us to show
by induction  that \begin{equation}\label{small} s_j <
\tfrac{a+b}{a-b} \quad \hbox{for all} \   j \ \hbox{satisfying} \
\ -\ell\leq j \leq -q-2\,.\end{equation} Indeed, $s_{-\ell} \leq 1
< \tfrac{a+b}{a-b}$ and for all $j =-\ell+1, -\ell+2, \ldots,
-q-2$, we have
$$ s_j < 1+ \tfrac{b}{a} +\tfrac{b}{a}s_{j-1} <  1+
\tfrac{b}{a} +\tfrac{b}{a} \tfrac{a+b}{a-b} =
\tfrac{a+b}{a-b}\,.$$ Combining \eqref{small} for $j= -q-2$  and
\eqref{big}, we get
$$
\tfrac{a+b}{a-b} > s_{-q-2} \geq \left(\tfrac{a}{b}\right)^q \; ,
\quad \hbox{ a contradiction with the choice of $q$}.$$
\end{description}
Both  discussed cases lead  to contradiction, therefore a
$p$-local function $\varphi$ converting in base $\beta$  from the
alphabet  $\mathcal{A}$ to  $\mathcal{B}$ cannot exist.

\end{proof}

\begin{proposition}\label{non-local1}
In base $\beta=a/b$, with $a$ and $b$ co-prime positive integers
such that $a > b \ge 1$, parallel addition is not possible on any
alphabet $\{-d, \ldots,0,\ldots, a+b-d-2 \}$ for $ 1 \leq d \leq
a+b-3$, of cardinality $a+b-1$.
\end{proposition}

\begin{proof}
If parallel addition was possible on $\{-d, \ldots,0,\ldots,
a+b-2-d\}$, then, by Proposition~\ref{conversion}, the conversion
$\varphi$ from $\{-d-1, \ldots, 0, \ldots, a+b-1-d\}$ to $\{-d,
\ldots, 0, \ldots, a+b-2-d\}$ would be a $p$-local function for
some $p$. The proof is then analogous to that of
Proposition~\ref{non-loc}, by considering the words
$\lexp{\omega}0(a-1-d)^n\bullet(a+b-1-d)(a+b-2-d)^q0^\omega$ and
$\lexp{\omega}0(a-1-d)^n\bullet 0^\omega$.
\end{proof}

Summarizing the results for both the cases of alphabets, either
with non-negative digits only, or with positive as well as
negative digits, we have proved that:

\begin{theorem}
Let $\beta = a/b$ be the base, with $a$ and $b$ co-prime positive integers such that $a > b \geq 1$, and let $\A$ be an alphabet (of contiguous integers containing~$0$) of the minimal cardinality allowing parallel addition in base $\beta$.\\
Then, $\A$ has cardinality $a+b$, and $\A$ has the form
\begin{itemize}
    \item $\A = \{0, \ldots, a+b-1\}$, or $\A = \{-a-b+1, \ldots, 0\}$, or
    \item $\A = \{-d, \ldots, 0, \ldots, a+b-1-d\}$ containing a subset $\{-b, \ldots, 0, \ldots, b\}$.
\end{itemize}
\end{theorem}

When $b=1$, we find back the classical case of positive integer
base, see Section~\ref{positive}.


\subsection{Negative Rational Base}\label{rat-neg}


\begin{proposition}\label{rat-par-neg}
Parallel addition in base $\beta=-a/b$, with $a$ and $b$  co-prime
positive integers such that $a > b \ge 1$, is possible on any
alphabet (of contiguous integers) of the form $\mathcal{A}_{-d} =
\{ -d, \ldots, 0, \ldots, a+b-1-d \}$ with cardinality
$\#\mathcal{A}_{-d} = a+b$, where $d \in \{ 0, \ldots, a+b-1 \}$.
\end{proposition}

\begin{proof}
Firstly, we show that parallel addition is possible on the
alphabet $\mathcal{A}_{0} = \{ 0, \ldots, a+b-1 \}$, by providing
an algorithm for the greatest digit elimination
\textsl{GDE}($-a/b$): $\{ 0, \ldots, a+b \} \rightarrow \{ 0,
\ldots, a+b -1\}$:

\vskip0.3cm \hrule \vskip0.2cm

\noindent {\bf Algorithm~\textsl{GDE}($-a/b$)}:  Base $\beta =
-a/b$, with $a > b \geq 1$, $a$ and $b$ co-prime positive
integers, parallel conversion (greatest digit elimination) from
$\{ 0, \ldots, a+b \}$ to $\{ 0, \ldots, a+b -1\}$.

\vskip0.2cm \hrule \vskip0.2cm

\noindent{\sl Input}: a finite sequence of  digits $(z_j)$ of $\{ 0, \ldots, a+b \}$, with $z = \sum z_j \beta^j$.\\
{\sl Output}: a finite sequence of digits $(z_j)$ of $\{ 0,
\ldots, a+b -1\}$, with $z = \sum z_j \beta^j$.

\vskip0.2cm

\noindent\texttt{for each $j$ in  parallel do}\\
1. \hspace*{0.5cm} \texttt{case}
    $\left\{\begin{array}{l}
        z_j = a+b \ \\
        a \le z_j \le a+b-1 \ \hbox{ \texttt{and}} \ 0 \le z_{j-1} \le b-1   \\
    \end{array} \right\} $\ \texttt{then} \, $q_j:=1$\\
   \hspace*{1cm} \texttt{if} \quad \ \,\,\, $0 \le z_j \le b-1  $ \texttt{and} \ $a \le z_{j-1} \le a+b $ \qquad \quad \ \texttt{then} \, $q_j:=-1$\\[2mm]
   \hspace*{1cm} \texttt{else} \qquad \qquad \qquad \qquad \qquad \qquad \qquad \qquad \qquad \qquad \qquad \quad $q_j:=0$\\
2. \hspace*{0.5cm} $z_j := z_j - a q_j - b q_{j-1}$

\vskip0.2cm \hrule \vskip0.2cm


Using our usual notion of $w_j = z_j - a q_j$, and $z_j^{new} =
w_j - b q_{j-1}$, we describe the various cases that can occur
during the course of this algorithm:
\begin{itemize}
    \item If $z_j = a+b$, we obtain $w_j = b$, and then $z_j^{new} \in b - b \cdot \{ -1, 0, 1 \} = \{ 0, b, 2b \} \subset \{ 0, \ldots, a+b-1 \} = \mathcal{A}_0$.
    \item For $z_j \in \{ a, \ldots, a+b-1 \} $ and $z_{j-1} \in \{ 0, \ldots, b-1 \}$, we have $q_j = 1$, and consequently $w_j \in \{ 0, \ldots, b-1 \}$. As $q_{j-1} \in \{ -1, 0 \}$, we finally get $z_j^{new} \in \{ 0, \ldots, b-1 \} - b \cdot \{ -1, 0 \} = \{ 0, \ldots, 2b-1 \} \subset \{ 0, \ldots, a+b-2 \} \subset \mathcal{A}_0$.
    \item For $z_j \in \{ a, \ldots, a+b-1 \} $ and $z_{j-1} \in \{ b, \ldots, a+b \}$, we have $q_j = 0$, so we keep $w_j \in \{ a, \ldots, a+b-1 \}$. Now $q_{j-1} \in \{ 0, 1 \}$, and thus $z_j^{new} \in \{ a, \ldots, a+b-1 \} - b \cdot \{ 0, 1 \} = \{ a-b, \ldots, a+b-1 \} \subset \{ 1, \ldots, a+b-1 \} \subset \mathcal{A}_0$.
    \item In the case of $z_j \in \{ b, \ldots, a-1 \}$, simply $q_j = 0$, $w_j \in \{ b, \ldots, a-1 \}$, and the resulting $z_j^{new} \in \{ b, \ldots, a-1 \} - b \cdot \{ -1, 0, 1 \} \subset \{ 0, \ldots, a+b-1 \} = \mathcal{A}_0$.
    \item When $z_j \in \{ 0, \ldots, b-1 \}$ and $z_{j-1} \in \{ 0, \ldots, a-1 \}$, we have $q_j = 0$, so we keep $w_j \in \{ 0, \ldots, b-1 \}$, and $q_{j-1} \in \{ -1, 0 \}$. Therefore, we obtain $z_j^{new} \in \{ 0, \ldots, b-1 \} - b \cdot \{ -1, 0 \} = \{ 0, \ldots, 2b-1 \} \subset \{ 0, \ldots, a+b-2 \} \subset \mathcal{A}_0$.
    \item Lastly, when $z_j \in \{ 0, \ldots, b-1 \}$ and $z_{j-1} \in \{ a, \ldots, a+b \}$, by means of $q_j = -1$ we get $w_j \in \{ a, \ldots, a+b-1 \}$. As $q_{j-1} \in \{ 0, 1 \}$, then $z_j^{new} \in \{a, \ldots, a+b-1\} - b \cdot \{ 0, 1 \} = \{ a-b, \ldots, a+b-1 \} \subset \mathcal{A}_0$.
\end{itemize}

Again, we must not forget to mention that the digit $z_j = 0$ is
transformed by this algorithm onto another digit (by means of $q_j
\neq 0$) only if its neighbour $z_{j-1}$ is a non-zero, namely
from the set $\{ a, \ldots, a+b \}$; thereby, it is ensured that
the algorithm cannot assign a string of non-zeros to a string of
zeros. The output value $z$ is equal to the input value $z$, since
the base $\beta$ fulfils the equality $b \beta^{j+1} + a \beta^j =
0$ for any $j \in \mathbb{Z}$. Thus, we see that this algorithm is
correct for the greatest digit elimination from the alphabet
$\mathcal{A}_0 \cup \{ a+b \}$ into $\mathcal{A}_0 = \{ 0, \ldots, a+b-1 \}$.\\

Now let us point out that all the elements $d \in \{ 0, \ldots,
a+b-1 \}$ are fixed by the $p$-local function $\varphi$ given by
this algorithm, in the sense that $\varphi(^{\omega}d \bullet
d^{\omega}) = (^{\omega}d \bullet d^{\omega})$. This fact,
together with Corollary~\ref{PossibleShifts}, implies that
parallel addition in the negative rational base $\beta = -a/b$ is
possible on any alphabet of the form $\mathcal{A}_{-d} = \{ -d,
\ldots, 0, \ldots, a+b-1-d \}$, with cardinality $\#
\mathcal{A}_{-d} = a+b$.
\end{proof}

Also here in the negative case $\beta=-a/b$, for $b=1$ we find
back the classical case of (negative) integer base, see
Section~\ref{negative}.

\bigskip

Since we do not have any lower bound for this base, we must find
one directly.

\begin{proposition}\label{rat-par-negMIN}  Let $\mathcal{A} =
\{m, \ldots, 0, \ldots,  M\}$ with $m\leq 0\leq M$ be an alphabet
of contiguous integers which enables parallel addition in base
$\beta=-a/b$, with $a$ and $b$ co-prime positive integers, $a > b
\ge 1$. Then $\#\mathcal{A}\geq a+b$.
\end{proposition}

\begin{proof}  Without loss of generality, we may assume  that  $1\leq M$. Let $\varphi$ be a $p$-local function realizing parallel
conversion  from $\mathcal{A}\cup\{M+1\}$ into $\mathcal{A}$ using
the mapping $\Phi: (\mathcal{A}\cup\{M+1\})^p \to \mathcal{A}$.
Put  $x=M+1$ and $y=\Phi(x^p)$.   According to Claim
\ref{divides1}, we have
\begin{equation}\label{multipl}
y-x= \left(-\tfrac{a}{b}-1\right) \sum_{k=0}^n c_k
\bigl(-\tfrac{a}{b}\bigr)^k\quad \text{for some }\ n \in
\mathbb{N}\ \text{and } \ \ c_k \in \mathbb{Z}.
\end{equation}
Multiplying the previous equation by $-b^{n+1}$ one gets
$$(x-y)b^{n+1} = (a+b)\sum_{k=0}^n c_k(-a)^kb^{n-k}, $$
and consequently, the number $a+b$ divides $(x-y)b^{n+1}$. As $a$
and $b$ are co-prime, necessarily  $a+b$ divides $x-y$. Since
$x-y>0$, there  exists $k \in \mathbb{N}$ such that $x-y =
k(a+b)\geq a+b$. But simultaneously,  $x-y\leq M+1-m = \#
\mathcal{A}$. Putting these two inequalities together, we obtain
$a+b\leq \#\mathcal{A}$.
\end{proof}
We can summarize this section into the following theorem.

\begin{theorem}
In base $\beta=-a/b$, with $a$ and $b$ co-prime positive integers,
$a > b \ge 1$, parallel addition can be performed in  any alphabet
$\mathcal{A}$ of contiguous integers containing $0$ with
cardinality  $\# \mathcal{A} =a+b$. This cardinality cannot be
reduced.
\end{theorem}


\section{Conclusions and comments}


Here is a summary of the numeration systems studied in this paper.
We have considered only alphabets of contiguous integers
containing $0$.

\bigskip

\begin{center}
\begin{tabular}{|l |l |p{6cm}|}
\hline
  Base & Canonical alphabet &  Minimal alphabet for parallel addition\\
\hline \hline
 $b \ge 2$ integer  &$\{0, \ldots,b-1\}$  & All alphabets of size $b+1$ \\
 \hline
 $-b$, $b \ge 2$ integer  &$\{0, \ldots,b-1\}$  & All alphabets of size $b+1$  \\

 \hline
 $\sqrt[k]{b}$, $b \ge 2$ integer  &  & All alphabets of size $b+1$ \\

 \hline
  $-1 + \imath$ & $\{0, 1\}$ & All alphabets of size $5$\\
  \hline
  $2 \imath$ & $\{0, \ldots,3\}$ & All alphabets of size $5$\\
  \hline
  $\imath \sqrt{2} $ & $\{0, 1\}$ & All alphabets of size $3$\\
  \hline
   $\beta^2=a \beta-1$ & $\{0, \ldots,a-1\}$ & All alphabets of size $a$\\
   \hline
  $\beta^2=a \beta+1$ & $\{0, \ldots,a\}$ & All alphabets of size $a+2$  \\
  \hline
   $a/b$ & $\{0, \ldots,a-1\}$ &  $\{0,\ldots, a+b -1\}$, $\{-a-b+1,\ldots, 0\}$, and
  all alphabets of size $a+b$ containing $\{-b, \ldots,0, \ldots,b\}$ \\
  \hline
   $-a/b$ &  $\{0, \ldots,a-1\}$ &   All alphabets of size $a+b$ \\
\hline
\end{tabular}
\end{center}

\bigskip

\medskip

Generalization of these results to other bases remains an open
problem. The cases of the Tribonacci numeration system with basis
satisfying the equation $X^3=X^2+X+1$, or quadratic bases
satisfying the equation $X^2=aX \pm b$, $b \ge 2$, are not so
straightforward. The reason is that we have only two tools so far,
namely Theorem~\ref{zdola} and Theorem~\ref{zdola2}, which provide
us with lower bounds on the cardinality of the alphabet. For the
bases listed in the summary table, the bounds given in these
theorems were attained, the only exception being the rational
bases $\beta = \pm a/b$, for which we had to refine our methods
specifically in order to prove minimality of the alphabets. These
examples show that, for attacking the question of minimality of
alphabet for other bases, we need to find stronger versions of
Theorem~\ref{zdola} and Theorem~\ref{zdola2}.

The positive rational base $\beta = a/b$ is exceptional among our
results by another property as well. Contrary to the other bases,
not all alphabets of contiguous integers (containing~$0$) with
sufficiently large cardinality allow parallel addition.

As mentioned in Remark \ref{kblock}, for alphabets which are too
small to allow parallel addition in a given base, one can consider
a more general concept of the so-called $k$-block $p$-local
functions. It means that, instead of a base $\beta$ and an
alphabet $\mathcal{A}$, we consider addition in base $\beta^k$ and
on the alphabet $\mathcal{A}_k=\{\sum_{j=0}^{k-1}
a_j\beta^j\,\vert \, a_j \in \mathcal{A}\}$. Our interest in
addition in  base $\beta$ can be extended to the question: What is
the minimal size of an alphabet $\mathcal{A}$ and the minimal size
$k$ of the blocks such that addition can be performed by a
$k$-block $p$-local function. This question was not tackled here
at all. In \cite{Kornerup}, the precise definition of $k$-block
$p$-local function and a relation between $\mathcal{A}$ and $k$
can be found.


\section*{Acknowledgements}


We acknowledge financial support by the Czech Science Foundation
grant GA\v{C}R 201/09/0584, and by the grants MSM6840770039 and
LC06002 of the Ministry of Education, Youth, and Sports of the
Czech Republic.



\end{document}